\documentclass[12pt, a4j]{amsart}
\title{On fibration stability after Dervan-Sektnan and singularities}
\author{Masafumi Hattori}
\address{Department of Mathematics, Faculty of Science, Kyoto University, Kyoto, 606-8502, Japan}
\email{hattori.masafumi.47z@st.kyoto-u.ac.jp}
\usepackage{amsthm}
\usepackage{amsmath}
\usepackage{amssymb}
\usepackage{amsfonts}
\usepackage{mathrsfs}
\usepackage{enumerate}
\usepackage{amscd}
\usepackage{graphicx}
\usepackage[all]{xy}
\usepackage[a4paper, top=2cm, bottom=2cm, left=2cm, right=2cm]{geometry}
\def\coloneq{\mathrel{\mathop:}=}
\newtheorem{thm}{Theorem}[section]
\newtheorem{lem}[thm]{Lemma}
\newtheorem{prop}[thm]{Proposition}
\newtheorem*{claim}{Claim}

\newtheorem*{thth}{Theorem}
\newtheorem{cor}[thm]{Corollary}
\newtheorem{teiri}{Theorem}

\newtheorem{kei}[teiri]{Corollary}

\theoremstyle{definition}
\newtheorem{de}[thm]{Definition}
\newtheorem{ex}[thm]{Example}

\newtheorem{conj}[thm]{Conjecture}
\newtheorem*{conjj}{Conjecture}
\newtheorem{rem}[thm]{Remark}
\newtheorem*{ntt}{Notation}

\begin{document}
\maketitle
\begin{abstract}
We introduce $\mathfrak{f}$-stability, a modification of fibration stability of Dervan-Sektnan \cite{DS4}, and show that $\mathfrak{f}$-semistable fibrations have only semi log canonical singularities. Moreover, $\mathfrak{f}$-stability puts restrictions on semi log canonical centers on Fano fibrations.
\end{abstract}
\tableofcontents
\section{Introduction}

It is one of the most important problems in K\"{a}hler geometry that when a constant scalar curvature K\"{a}hler (cscK) metric exists on a polarized complex manifold $(X,L)$. The Yau-Tian-Donaldson (YTD) conjecture predicts that the existence of cscK metrics is equivalent to a certain algebro-geometric condition called K-polystability. Indeed, Berman-Darvas-Lu \cite{BDL} proved that if $(X,L)$ admits a cscK metric, then $(X,L)$ is K-polystable and Chen-Donaldson-Sun \cite{CDS} and Tian \cite{T2} proved YTD conjecture in Fano case independently. On the other hand, K-stability is the positivity of the leading term of Chow weight and is also an important notion in terms of the geometric invariant theory (GIT). Ross and Thomas \cite{RT07} studied K-stability in an algebro-geometric way first, and Odaka \cite{GIToda} found out the relationship between K-stability and singularities by applying the minimal model program (MMP). He also proved that if a $\mathbb{Q}$-Gorenstein variety $V$ is asymptotically Chow-semistable, then $V$ has only slc singularities as a corollary.

On the other hand, the existence problem of cscK metrics on fibrations is well-studied in K\"{a}hler geometry. In this paper, fibrations mean algebraic fiber spaces $f:X\to B$, where $f$ are morphisms of varieties with connected general fiber. See the details in Definition \ref{algfib}. We call an algebraic fiber space $f:X\to B$ is 
\begin{itemize}
\item a {\it Calabi-Yau fibration} if there exists a line bundle $L_0$ on $B$ such that $K_X\sim_{\mathbb{Q}} f^*L_0$.
\item a {\it Fano fibration} (resp., {\it a canonically polarized fibration}) if $K_X$ is $f$-antiample (resp., $f$-ample).
\item a {\it smooth} (resp., {\it flat}) {\it fibration} if $f$ is smooth (resp., flat).
\end{itemize}

Fine obtained a sufficiency condition \cite[Theorem 1.1]{Fine} for existence of a cscK metric on a smooth fibration whose fibers have cscK metrics first. Dervan and Sektnan \cite{DS2} introduced a differential geometric notion called optimal symplectic connection and proved the following generalization of the result of Fine:
\begin{thth}[{\cite[Theorem 1.2]{DS2}}]
Let $f:(X,H)\to (B,L)$ be a polarized smooth fibration. If $f$ admits an optimal symplectic connection and $(B,L)$ admits a twisted cscK metric with respect to the Weil-Petersson metric, then $(X,\delta H+f^*L)$ has cscK metrics for sufficiently small $\delta>0$.
\end{thth}

This is the definitive result for smooth fibrations. On the other hand, Jian-Shi-Song \cite{JSS} proved that smooth good minimal models (i.e., $K_X$ is semiample) have cscK metrics by the result of Chen-Cheng \cite{Ch} and J-stability (cf., \cite{SW}). Here, we emphasize that they treated not only smooth fibrations but Calabi-Yau fibrations admitting a singular fiber. Sj\"{o}str\"{o}m Dyrefelt \cite{Sj} and Song \cite{S} generalized independently the result of Jian-Shi-Song to the case when $X$ is a smooth minimal model (i.e., $K_X$ is nef) later. On the other hand, the author proved K-stability of klt minimal models in \cite{Hat}.

Dervan and Sektnan \cite{DS4} also conjectured that the existence of an optimal symplectic connection is equivalent to an algebro-geometric condition called fibration stability. Their theorem and conjecture predict that fibration stability and a certain stability of the base variety imply adiabatic K-stability (cf., Definition \ref{algfibk}) of the total space as follows.
\begin{conjj}[{\cite[1.3]{DS4}}]
Let $f:(X,H)\to (B,L)$ be a smooth fibration. If $f$ is fibration stable and $(B,L)$ has a twisted cscK metric with respect to the Weil-Petersson metric, then $(X,H)$ is adiabatically K-semistable i.e., $(X,\delta H+f^*L)$ is K-semistable for sufficiently small $\delta>0$.
\end{conjj}
 If it was true, we could obtain a criterion for K-stability of fibrations. The condition on the base in the conjecture would be necessary. Indeed, it is known that twisted K-stability of the base is necessary for adiabatic K-stability of the total space by the work of Dervan-Ross \cite[Corollary 4.4]{DR2}.
  
In this paper, we introduce $\mathfrak{f}$-stability (cf., Definition \ref{ds1}) as a modification of fibration stability of Dervan-Sektnan \cite{DS4} and prove fundamental results on this. This is a stronger condition than fibration stability in numerical aspects. Dervan and Sektnan weakened the definition of fibration stability in \cite{DS4}. On the other hand, the original one in \cite{DS} is the condition of the positivity of $W_0$ and $W_1$ we will explain in the next page. Note that the original fibration stability coincides with $\mathfrak{f}$-stability for fibrations over curves such that each fiber is K-polystable. Since any polarized smooth curve is twisted K-stable, it is natural to conjecture that the original fibration stability implies existence of an optimal symplectic connection on fibrations over smooth curves rather than the new one. More generally, for a smooth fibration $f:(X,H)\to (B,L)$, if $B$ has a twisted cscK metric in the sense of the conjecture above and $X$ has an optimal symplectic connection, it is easy to see that $f$ is $\mathfrak{f}$-semistable by \cite[Theorem 1.2]{DS2}. Taking these facts into account, it is worth studying $\mathfrak{f}$-stability for K-stability of fibrations.
\vspace{0mm}

First, we introduce invariants $W_i$ as the Donaldson-Futaki invariant for $\mathfrak{f}$-stability. Similarly to results of \cite{GIToda} in K-stability, we establish the explicit formula to compute $W_i$ by taking general hyperplane sections of the base, and show that $\mathfrak{f}$-stability puts some restrictions on singularities by applying MMP. We obtain the following.

\begin{teiri}[Theorem \ref{singular}]\label{singular1}
Let $f:(X,\Delta,H)\to (B,L)$ be a polarized algebraic fiber space pair (cf., Definition \ref{algfib}). If $f $ is $\mathfrak{f}$-semistable, $(X,\Delta)$ has at most lc singularities.
\end{teiri}

We remark that $\mathfrak{f}$-stability does not imply adiabatic K-stability in general. Indeed, a rational elliptic surface with a section and a $II^*$, $III^*$ or $IV^*$-fiber is $\mathfrak{f}$-stable but adiabatic K-unstable over $\mathbb{P}^1$. We prove this fact in \cite{Hat2}. Thus, we can not apply Odaka's result \cite{GIToda} to Theorem \ref{singular1} directly.

  We can also show that a $\mathfrak{f}$-semistable flat Fano fibration is Kawamata log terminal (klt) if the base variety has only klt singularities. Moreover, as an application of $\mathfrak{f}$-stability to K-stability, we show that adiabatically K-semistable flat Fano fibrations over klt varieties have only klt singularities. 
  
 \begin{teiri}[Theorem \ref{Fanofib}]\label{Fanofib1}
Let $f:(X,H)\to (B,L)$ be a flat polarized algebraic fiber space. Suppose that there exist $\lambda\in \mathbb{Q}_{>0}$ and a line bundle $L_0$ on $B$ such that $H+f^*L_0\equiv-\lambda\, K_{X}$ and $B$ has only klt singularities. If $f$ is $\mathfrak{f}$-semistable, then $X$ has only klt singularities. In particular, if $f$ is adiabatically K-semistable (Definition \ref{algfibk}), then $X$ has only klt singularities.
\end{teiri}

On the other hand, we prove that klt Calabi-Yau fibrations are $\mathfrak{f}$-stable similarly to the well-known theorem \cite[Theorem 2.10]{O2} in K-stability. Indeed, we show the following.

\begin{teiri}[Theorem \ref{kds}]\label{cc}
Let $f:(X,\Delta,H)\to (B,L)$ be a polarized algebraic fiber space pair with a line bundle $L_0$ on $B$. Suppose that $(X,\Delta,H)$ is a klt polarized pair and $K_X+\Delta\equiv f^*L_0$. Then $f$ is $\mathfrak{f}$-stable.
\end{teiri}

We explain the definition of $\mathfrak{f}$-stability briefly as follows. Suppose that $f:(X,H)\to (B,L)$ is a polarized algebraic fiber space. Let $m=\mathrm{rel.dim}\, f$ and $\mathrm{dim}\, B=n$. Roughly speaking, for any semiample test configuration $(\mathcal{X},\mathcal{H})$ for $(X,H)$, we define constants $W_0(\mathcal{X},\mathcal{H}), \cdots ,W_n(\mathcal{X},\mathcal{H})$ and a rational function $W_{n+1}(\mathcal{X},\mathcal{H})(j)$ in $j$ so that 
\[
V(H+jL)M^{\mathrm{NA}}(\mathcal{X},\mathcal{H}+jL)=W_{n+1}(\mathcal{X},\mathcal{H})(j) + \sum_{i=0}^n j^iW_{n-i}(\mathcal{X},\mathcal{H})
\]
where $\lim_{j\to\infty}W_{n+1}(\mathcal{X},\mathcal{H})(j)=0$ and $M^{\mathrm{NA}}$ is the non-Archimedean Mabuchi functional (cf., Definition \ref{nadef} and Notation in \S4). Then $f:(X,H)\to (B,L)$ is {\it $\mathfrak{f}$-semistable} if $$\sum_{i=0}^n j^iW_{n-i}(\mathcal{X},\mathcal{H})\ge 0$$ for sufficiently large $j>0$. We calculate $W_i$ in Lemma \ref{induction} as follows,
 \begin{align*}
W_{i} (\mathcal{X},\mathcal{H})&= \binom{n+m}{n-i} \Biggl(K^{\mathrm{log}}_{\mathcal{X}\cap D_1 \cap D_2 \cap \cdots \cap D_{n-i}/\mathbb{P}^1} \cdot \mathcal{H}_{|\mathcal{X}\cap D_1 \cap D_2 \cap \cdots \cap D_{n-i}}^{m+i}\\
&+\frac{S(X_b,H_b)}{m+i+1}\mathcal{H}_{|\mathcal{X}\cap D_1 \cap D_2 \cap \cdots \cap D_{n-i}}^{m+i+1} \Biggr)+\sum_{k=n-i+1}^n C_kJ^{\mathrm{NA}}(\mathcal{H}_{|\mathcal{X}\cap D_1 \cap D_2 \cap \cdots \cap D_{j}})
\end{align*}
where $C_k$ are constants and $D_k\in|L|$ are general elements. Then we prove Theorems \ref{singular1}, \ref{Fanofib1} by applying MMP results (cf., \cite[Theorem 1.1]{OX}, \cite{HX} and the adjunction formula \cite[\S16, \S17]{K+}) and Lemma \ref{induction}. On the other hand, we also show these theorems for deminormal pairs with boundaries in \S\S \ref{nnc}. More precisely,
\begin{teiri}[Theorem \ref{dingular}, Theorem \ref{dFanofib}]\label{ddd}
Suppose that $f:(X,\Delta,H)\to (B,L)$ is a polarized deminormal algebraic fiber space pair (Definition \ref{dalgfib}). 
\begin{enumerate}
\item If $f $ is $\mathfrak{f}$-semistable, $(X,\Delta)$ has at most slc singularities. 
\item Suppose that there exist $\lambda\in \mathbb{Q}_{>0}$ and a line bundle $L_0$ on $B$ such that $H+f^*L_0\equiv-\lambda\, (K_{X}+\Delta)$, and $f$ is $\mathfrak{f}$-semistable. Then any slc-center $C$ of $(X,\Delta)$ is of fiber type, i.e., $\mathrm{codim}_XC\le\mathrm{codim}_Bf(C)$ (Definition \ref{fibertype}).
\end{enumerate}
\end{teiri}

Note that if $f$ is flat, $C$ is of fiber type iff $C$ contains an irreducible component of a fiber. Thus, Theorem \ref{Fanofib1} follows from Theorem \ref{ddd}. On the other hand, $\mathfrak{f}$-semistability is a weaker condition than adiabatic K-semistability. Therefore, we also obtain the following.
\begin{kei}[Corollary \ref{kei56}]
 Let $f:(X,\Delta,H)\to (B,L)$ be an adiabatically K-semistable polarized deminormal algebraic fiber space pair such that $H+f^*L_0\equiv-\lambda\, (K_{X}+\Delta)$ where $L_0$ is a line bundle on $B$ and $\lambda\in\mathbb{Q}_{>0}$. Then, $(X,\Delta)$ is slc and any slc-center of $(X,\Delta)$ is of fiber type.
\end{kei}

\subsection*{Outline of this paper}
In \S2, we prepare many terminology and facts on K-stability and algebraic fiber spaces. In \S3, we recall results of Odaka \cite{GIToda} and Boucksom-Hisamoto-Jonsson \cite{BHJ} on the relationship between singularities and K-stability. From \S4, we state our original results. In \S4, we calculate $W_i$ and deduce the Theorem \ref{cc}. In \S\S5.1, we apply results of MMP and the computations in \S4 to obtain Theorem \ref{singular1} and a generalization (Theorem \ref{genFanofib}) of Theorem \ref{Fanofib1}. In \S\S 5.2, we extend theorems in \S\S 5.1 to the deminormal case.

\section*{Acknowledgements}
The author can not express enough his sincere and deep gratitude to Professor Yuji Odaka who is his research advisor for a lot of suggestive advices, productive discussions and reading the draft. The author also would like to thank Professor Ruadha\'{i} Dervan for helpful comments and warm encouragement. This work is a part of my master thesis.

\section{Notation}
\label{Notat}
In this paper, we work over $\mathbb{C}$. If $X$ is a {\it scheme}, we assume that $X$ is a scheme of finite type over $\mathbb{C}$ in this paper. If $X$ is a {\it variety}, we assume that $X$ is an irreducible, reduced and separated scheme. We follow the definitions of $\mathbb{Q}$-line bundles, $\mathbb{Q}$-Weil divsiors and $\mathbb{Q}$-Cartier divisors, and the notations of the $\mathbb{Q}$-linearly equivalence $\sim_{\mathbb{Q}}$ and the numerical equivalence $\equiv$ from \cite{Ha}, \cite{KoMo} and \cite{BHJ}. A pair $(X,L)$ is a {\it polarized scheme} if $X$ is a proper, reduced and equidimensional scheme over $\mathbb{C}$ and $L$ is an ample $\mathbb{Q}$-line bundle over $X$. A $\mathbb{Q}$-Weil divisor $\Delta$ such that there exist integral divisors $F_i$ different from each other, $a_i\in\mathbb{Q}$ and $n\in \mathbb{Z}_{\ge 0}$ such that $\Delta=\sum_{i=1}^n a_iF_i$ is a {\it boundary} if $K_{(X,\Delta)}\coloneq K_X+\Delta$ is $\mathbb{Q}$-Cartier on a normal variety $X$. Then, we call a pair $(X,\Delta)$ a (normal) log pair. Here, we do not assume that $\lceil\Delta\rceil=\sum\lceil a_i\rceil F_i$ is a reduced divisor. A $\mathbb{Q}$-divisor $\Delta=\sum_{i=1}^r a_i D_i$ on a smooth variety $X$ is {\it simple normal crossing} (snc) if each $\bigcap_{i\in J}D_i$ is smooth for any subset $J\subset \{1,2,\cdots,r\}$.

First, recall the definitions of deminormal pairs and semi log canonical singularities,
\begin{de}[deninormal pair {\cite[Chapter 5]{Ko}}]\label{definition1}
Let $ X$ be an equidimensional reduced scheme satisfying the Serre's condition $S_2$. $X$ is a {\it deminormal scheme} if any codimension 1 point of $X$ is smooth or nodal. A $\mathbb{Q}$-divisor $\Delta$ on $X$ is a boundary if $K_X+\Delta$ is $\mathbb{Q}$-Cartier and any irreducible component of $\Delta$ is not contained in the singular locus $\mathrm{Sing}(X)$. Then we call $(X,\Delta)$ a {\it deminormal log pair}.

Let $\nu:\tilde{X}\to X$ be the normalization. Then, $\mathfrak{d}=\mathrm{Hom}_{\mathcal{O}_{X}}(\nu_*\mathcal{O}_{\tilde{X}},\mathcal{O}_X)$ is an ideal in both of $\mathcal{O}_X$ and $\mathcal{O}_{\tilde{X}}$. We call the {\it conductors} of $X$ the closed subschemes defined by $\mathfrak{d}$ and we will denote these by $\mathfrak{cond}_X\subset X$ and $\mathfrak{cond}_{\tilde{X}}\subset \tilde{X}$ respectively. Note that $\mathfrak{cond}_{\tilde{X}}$ is a Weil divisor. 

Further, let $L$ be a $\mathbb{Q}$-ample line bundle on $X$. Then, we call $(X,\Delta,L)$ a {\it polarized deminormal (log) pair}.
\end{de}

\begin{de}[Log discrepancy]
Let $(X,\Delta)$ be a normal log pair and $v$ be a divisorial valuation on $X$. Suppose that $\sigma:Y\to X$ be a birational morphism such that there exist a positive constant $c>0$ and a prime divisor $F$ on $Y$ such that $v=c\,\mathrm{ord}_F$. Then the {\it log discrepancy} of $v$ with respect to $(X,\Delta)$ is
\[
A_{(X,\Delta)}(v)=v(K_Y-\sigma^*K_{(X,\Delta)})-c.
\] 
It is easy to see that the log discrepancy is independent of $\sigma$. We define $A_{(X,\Delta)}(v_{\mathrm{triv}})=0$ if $v_{\mathrm{triv}}$ is the trivial valuation.

We define $(X,\Delta)$ is 
\begin{itemize}
\item {\it sub Kawamata log terminal (subklt)} if $A_{(X,\Delta)}(v)>0$ for any non trivial divisorial valuation $v$.
\item {\it sub log canonical (sublc)} if $A_{(X,\Delta)}(v)\ge0$.
\end{itemize}
Let $c_X(v)$ be the center of $v$ on $X$. $c_X(v)$ is an lc center of $(X,\Delta)$ if $A_{(X,\Delta)}(v)=0$. $(X,\Delta)$ is klt (resp., lc) if $(X,\Delta)$ is subklt (resp., sublc) and $\Delta$ is effective. We also say that $(X,\Delta)$ has only klt (resp., lc) singularities.

Let $(V,B)$ be a deminormal log pair and $\nu:V'\to V$ be the normalization. Then, $(V,B)$ is {\it semi log canonical (slc)} if $(V',\nu^{-1}_*B+\mathfrak{cond})$ is lc.
\end{de}

Recall the notion of the minimal model program (MMP). We follow the fundamental notations of MMP in \cite{KoMo}.

\begin{de}[Minimal model]
Let $(X,\Delta)$ be a projective normal log pair over a quasi projective normal variety $S$ such that $\lceil\Delta\rceil$ is reduced. Let $Y$ be a projective normal variety over $S$. Then, a birational map $\phi:X\dashrightarrow Y$ is a {\it birational contraction} if there is no $\phi^{-1}$-exceptional divisor. Suppose that $K_Y+\phi_*\Delta$ is also $\mathbb{Q}$-Cartier. Then, $\phi$ is $(K_X+\Delta)$-non-positive (resp., $(K_X+\Delta)$-negative) if we have $p^*(K_X+\Delta)=q^*(K_Y+\phi_*\Delta)+E$ where $E$ is effective (resp., $E$ is effective and $\mathrm{Supp}\, E$ contains all $\phi^{-1}$-exceptional divisors). Here, $\Gamma$ is the resolution of singularities of the graph of $\phi$ and $p:\Gamma\to X$ and $q:\Gamma\to Y$ are canonical projections. Now $Y$ is
\begin{itemize}
\item a {\it weak log canonical model} (wlcm) of $(X,\Delta)$ if $K_Y+\phi_*\Delta$ is nef over $S$ and $\phi$ is $(K_X+\Delta)$-non-positive,
\item a {\it minimal model} (resp., {\it good minimal model}) of $(X,\Delta)$ if $K_Y+\phi_*\Delta$ is nef (resp., semiample) over $S$ and $\phi$ is $(K_X+\Delta)$-negative,
\item the {\it log canonical model} (lc model) if $K_Y+\phi_*\Delta$ is ample over $S$ and $Y$ is a wlcm.
\end{itemize}
\end{de}

 Next, recall K-stability. For example, see the detail in \cite{BHJ}.

\begin{de}[Test configuration]
Let $X$ be a proper scheme. Then, $\pi:\mathcal{X}\to \mathbb{A}^1$ is a {\it test configuration} for $X$ if $\pi:\mathcal{X}\to \mathbb{A}^1$ satisfies the following properties:
\begin{enumerate}
\item $\mathcal{X}$ has a $\mathbb{G}_m$-action. 
\item $\pi$ is a proper, flat and $\mathbb{G}_m$-equivariant morphism where $\mathbb{A}^1$ admits a canonical $\mathbb{G}_m$-action.
\item $\mathcal{X}_1\coloneq\pi^{-1}(1)\cong X$.
 \end{enumerate}
If there is no fear of confusion, we will denote $\pi:\mathcal{X}\to \mathbb{A}^1$ as $\mathcal{X}$ simply. Let $(X,L)$ be a polarized deminormal scheme. Then, $(\mathcal{X},\mathcal{L})$ is a (semi)ample test configuration for $(X,L)$ if
 \begin{enumerate}
\item $\mathcal{X}$ is a test configuration for $X$. 
\item $\mathcal{L}$ is a $\mathbb{A}^1$-(semi)ample $\mathbb{G}_m$-equivariant $\mathbb{Q}$-line bundle.
\item $\mathcal{L}|_{\mathcal{X}_1}=L$.
 \end{enumerate}

 $(X_{\mathbb{A}^1},L_{\mathbb{A}^1})=(X\times\mathbb{A}^1,L\times\mathbb{A}^1)$ with a trivial $\mathbb{G}_m$-action is a semiample test configuration. We call this the {\it trivial test configuration}. Note that any test configuration is birational. $(\mathcal{X},\mathcal{L})$ is isomorphic to $(X_{\mathbb{A}^1},L_{\mathbb{A}^1})$ in codimension 1, then we call $(\mathcal{X},\mathcal{L})$ almost trivial.
 
 $\mathcal{X}$ {\it dominates} $X_{\mathbb{A}^1}$ if there exists a birational morphism of test configurations $\rho_{\mathcal{X}}:\mathcal{X}\to X_{\mathbb{A}^1}$.
 
 $(\mathcal{X},\mathcal{L})$ is a {\it product} test configuration if $\mathcal{X}$ is isomorphic to $X\times \mathbb{A}^1$ as abstract varieties.
 
 For simplicity, the {\it central fiber} of $(\mathcal{X},\mathcal{L})$ is the fiber of $\pi$ over $0\in\mathbb{A}^1$ and is denoted as $(\mathcal{X}_0,\mathcal{L}_0)$. 
 
\end{de}

To define Donaldson-Futaki invariant, we need the following definition of the weight of $\mathbb{G}_m$-representation. Let $W$ be a finite dimensional $\mathbb{G}_m$-representation space over $\mathbb{C}$ and then $W$ has the unique weight decomposition
\[
W=\bigoplus_{k=-\infty}^\infty W_k
\]
where $\lambda\in\mathbb{G}_m$ acts on $v\in W_k$ in the way that $v\mapsto \lambda^kv$. Then, the {\it weight} of $W$ is
\[
\sum_k -k\,\mathrm{dim}\, W_k.
\]

\begin{de}[Donaldson-Futaki invariant, \cite{OS} Definition 3.2]
Let $(X,\Delta,L)$ be an $n$-dimensional polarized (demi)normal pair and $(\mathcal{X},\mathcal{L})$ be an ample test configuration. Let $w(m)$ be the weight of $H^0(\mathcal{X}_0,m\mathcal{L}_0)$. It is well-known that $w(m)=\sum_{i=0}^{n+1}b_im^{n+1-i}$ is a polynomial function in $m\gg0$ with $\mathrm{deg}\, w=n+1$ (see \cite[Theorem 3.1]{BHJ}). On the other hand, let $N_m=h^0(X,mL)=\sum_{i=0}^{n}a_im^{n-i}$ be the Hilbert polynomial. Then, the {\it Donaldson-Futaki invariant} of $(\mathcal{X},\mathcal{L})$ is 
\[
\mathrm{DF}(\mathcal{X},\mathcal{L})=2\frac{b_1a_0-a_1b_0}{a_0^2}.
\]
Moreover, let $\Delta_{\mathcal{X}}$ be the strict transformation of $\Delta\times\mathbb{A}^1$. Let $\hat{w}(m)=\sum_{i=0}^{n}\hat{b}_im^{n-i}$ be the weight polynomial of $H^0(\Delta_{\mathcal{X},0},m\mathcal{L})|_{\Delta_{\mathcal{X},0}}$ and $\hat{N}_m=h^0(\Delta,mL|_{\Delta})=\sum_{i=0}^{n-1}\hat{a}_im^{n-1-i}$ be the Hilbert polynomial. Then, the {\it log Donaldson-Futaki invariant} of $(\mathcal{X},\mathcal{L})$ is 
\[
\mathrm{DF}_{\Delta}(\mathcal{X},\mathcal{L})=\mathrm{DF}(\mathcal{X},\mathcal{L})+\frac{\hat{b}_0a_0-\hat{a}_0b_0}{a_0^2}.
\]
\end{de}

\begin{ntt}
In this paper, we write line bundles and divisors interchangeably. For example, $$L+mH=L\otimes H^{\otimes m}.$$ For simplicity, we will denote intersection products of line bundles or divisors as $$L^m\cdot (H+D)=L^{\cdot m}\cdot (H\otimes \mathcal{O}_X(D)).$$ 
\end{ntt}

We define the non-Archimedean functionals as in \cite{BHJ}.

\begin{de}[{\cite[\S 6, 7]{BHJ}}]\label{nadef}
Let $(X,\Delta,L)$ be an $n$-dimensional polarized normal pair and $\pi:(\mathcal{X},\mathcal{L})\to\mathbb{A}^1$ be a normal semiample test configuration. Let also $(\overline{\mathcal{X}},\overline{\mathcal{L}})$ be the {\it $\mathbb{G}_m$-equivariant compactification over $\mathbb{P}^1$} such that the $\infty$-fiber $(\overline{\mathcal{X}}_\infty,\overline{\mathcal{L}}_\infty)$ is $\mathbb{G}_m$-equivariantly isomorphic to $(X,L)$ with a trivial $\mathbb{G}_m$-action (cf., \cite[Definition 2.7]{BHJ}). Suppose that there exists a $\mathbb{G}_m$-equivariant morphism $\rho:\overline{\mathcal{X}}\to X_{\mathbb{P}^1}$ such that $\rho$ is the identity on $X\times\mathbb{P}^1\setminus0$. Here, $(X_{\mathbb{P}^1},L_{\mathbb{P}^1})$ is the $\mathbb{G}_m$-equivariant compactification of $(X_{\mathbb{A}^1},L_{\mathbb{A}^1})$. We call $\rho$ the {\it canonical} map. We also denote $\rho^*L_{\mathbb{P}^1}$ by $L_{\mathbb{P}^1}$. If
\begin{itemize}
\item $V=V(L)\coloneq L^n$,
\item $S(X,\Delta,L)=-\frac{n(K_{(X,\Delta)}\cdot L^{n-1})}{(L^n)}$ where $K_{(X,\Delta)}=K_X+\Delta$,
\item $K^{\mathrm{log}}_{(\overline{\mathcal{X}},\Delta_{\overline{\mathcal{X}}})/\mathbb{P}^1}=K_{(\overline{\mathcal{X}},\Delta_{\overline{\mathcal{X}}})/\mathbb{P}^1}+(\mathcal{X}_{0,\mathrm{red}}-\mathcal{X}_0)$ where $\mathcal{X}_{0,\mathrm{red}}$ is the reduced central fiber of $\pi$ and $\Delta_{\overline{\mathcal{X}}}$ is the strict transformation of $\Delta\times\mathbb{P}^1$ in $\overline{\mathcal{X}}$,
\item For any irreducible component of $\mathcal{X}_0$, the divisorial valuation $v_E$ on $X$ is the restriction of $b_E^{-1}\mathrm{ord}_E$ to $X$ where $b_E=\mathrm{ord}_E(\mathcal{X}_0)$ (cf., \cite[\S4]{BHJ}),
\end{itemize}
then
\begin{itemize}
\item the non-Archimedean Monge-Amp\`{e}re energy of $(\mathcal{X},\mathcal{L})$ is
\[
E^{\mathrm{NA}}(\mathcal{X},\mathcal{L})=\frac{\overline{\mathcal{L}}^{n+1}}{(n+1)V(L)},
\]
\item the non-Archimedean $I$-functional of $(\mathcal{X},\mathcal{L})$ is
\[
I^{\mathrm{NA}}(\mathcal{X},\mathcal{L})=V^{-1}(\overline{\mathcal{L}}\cdot L_{\mathbb{P}^1}^n)-V^{-1}(\overline{\mathcal{L}}- L_{\mathbb{P}^1})\cdot\overline{\mathcal{L}}^n,
\]
\item the non-Archimedean $J$-functional of $(\mathcal{X},\mathcal{L})$ is
\[
J^{\mathrm{NA}}(\mathcal{X},\mathcal{L})=V^{-1}(\overline{\mathcal{L}}\cdot L_{\mathbb{P}^1}^n)- E^{\mathrm{NA}}(\mathcal{X},\mathcal{L}),
\]
\item the $(\mathcal{J}^H)^\mathrm{NA}$-functional of $(\mathcal{X},\mathcal{L})$ is
\[
(\mathcal{J}^H)^{\mathrm{NA}}(\mathcal{X},\mathcal{L})=V^{-1}(H_{\mathbb{P}^1}\cdot \overline{\mathcal{L}}^{n})- V^{-1}(nH\cdot L^{n-1})E^{\mathrm{NA}}(\mathcal{X},\mathcal{L})
\]
where $H$ is a line bundle on $X$,
\item the non-Archimedean Ricci energy of $(\mathcal{X},\mathcal{L})$ is
\[
R_\Delta^{\mathrm{NA}}(\mathcal{X},\mathcal{L})=(\mathcal{J}^{K_{(X,\Delta)}})^{\mathrm{NA}}(\mathcal{X},\mathcal{L})-S(X,\Delta,L)E^{\mathrm{NA}}(\mathcal{X},\mathcal{L}),
\]
\item the non-Archimedean entropy of $(\mathcal{X},\mathcal{L})$ is (see \cite[Corollary 7.18]{BHJ})
\[
H_\Delta^{\mathrm{NA}}(\mathcal{X},\mathcal{L})=V^{-1}\sum_EA_{(X,B)}(v_E)(E\cdot\overline{\mathcal{L}}^n)=V^{-1}(K^{\mathrm{log}}_{(\overline{\mathcal{X}},\Delta_{\overline{\mathcal{X}}})/\mathbb{P}^1}-\rho^*K_{(X_{\mathbb{P}^1},\Delta_{\mathbb{P}^1})/\mathbb{P}^1})\cdot\overline{\mathcal{L}}^n
\]
where $E$ runs over the irreducible components of $\mathcal{X}_0$,
\item the non-Archimedean Mabuchi functional of $(\mathcal{X},\mathcal{L})$ is
\[
M_\Delta^{\mathrm{NA}}(\mathcal{X},\mathcal{L})=H_\Delta^{\mathrm{NA}}(\mathcal{X},\mathcal{L})+(\mathcal{J}^{K_{(X,\Delta)}})^{\mathrm{NA}}(\mathcal{X},\mathcal{L}).
\]
\end{itemize}
If there is no afraid of confusion, we denote $(\overline{\mathcal{X}},\overline{\mathcal{L}})$ as $(\mathcal{X},\mathcal{L})$.
Note that those functional are pullback invariant, if there exists a $\mathbb{G}_m$-equivariant morphism $\mu:\tilde{\mathcal{X}}\to\mathcal{X}$ of normal test configurations for $X$ such that $\mu$ is the identity on $X\times\mathbb{A}^1\setminus0$ and $F^{\mathrm{NA}}$ is one of the functionals as above, then 
\[
F^{\mathrm{NA}}(\mathcal{X},\mathcal{L})=F^{\mathrm{NA}}(\tilde{\mathcal{X}},\mu^*\mathcal{L}).
\]

On the other hand, a semiample test configuration $(\tilde{\mathcal{X}},\tilde{\mathcal{L}})$ {\it dominates} $(\mathcal{X},\mathcal{L})$ if there exists a $\mathbb{G}_m$-equivariant morphism $\mu:\tilde{\mathcal{X}}\to\mathcal{X}$ of normal test configurations for $X$ such that $\mu$ is the identity on $X\times\mathbb{A}^1\setminus0$ and $\tilde{\mathcal{L}}=\mu^*\mathcal{L}$. We define the set of all non-Archimedean positive metrics with respect to $(X,L)$ as follows,
\[
\mathcal{H}^{\mathrm{NA}}(L)=\{\mathrm{All}\, \mathrm{(semi)ample}\, \mathrm{test}\, \mathrm{configurations}\,\mathrm{for}\, (X,L)\}/\sim
\]
where $\sim$ is the equivalence relation generated by dominations. See also \cite[Proposition 2.17]{BHJ}. It is easy to see that $\mathcal{H}^{\mathrm{NA}}(L)$ is a set. Thus, the functional defined as above are well-defined on $\mathcal{H}^{\mathrm{NA}}(L)$. It is also easy to see that for any non-Archimedean positive metric $\phi\in\mathcal{H}^{\mathrm{NA}}(L)$ we can take a representative $(\mathcal{X},\mathcal{L})$ of $\phi$ such that $\mathcal{X}$ is normal and there exists a $\mathbb{G}_m$-equivariant morphism $\rho:\overline{\mathcal{X}}\to X_{\mathbb{P}^1}$ such that $\rho$ is canonical.
\end{de}

\begin{prop}[{\cite[Theorem 3.7]{OS}}, {\cite[Proposition 3.12]{BHJ}}]
Notations as in Definition \ref{nadef}. Let $(X,\Delta,L)$ be a polarized normal pair and $(\mathcal{X},\mathcal{L})$ be a normal semiample test configuration. Then, 
\[
\mathrm{DF}_{\Delta}(\mathcal{X},\mathcal{L})=V^{-1}(K_{(\overline{\mathcal{X}},\Delta_{\overline{\mathcal{X}}})/\mathbb{P}^1}-\rho^*K_{(X_{\mathbb{P}^1},\Delta_{\mathbb{P}^1})/\mathbb{P}^1})\cdot\overline{\mathcal{L}}^n+R^{\mathrm{NA}}_\Delta(\mathcal{X},\mathcal{L})+S(X,\Delta,L)E^{\mathrm{NA}}(\mathcal{X},\mathcal{L}).
\]
\end{prop}

\begin{de}[K-stability]
Let $(X,\Delta,L)$ be a polarized deminormal pair. $(X,L)$ is 
\begin{itemize}
\item {\it K-semistable} if 
\[
\mathrm{DF}_{\Delta}(\mathcal{X},\mathcal{L})\ge0
\]
for any semiample test configuration,
\item {\it K-polystable} if $(X,\Delta,L)$ is K-semistable and
\[
\mathrm{DF}_{\Delta}(\mathcal{X},\mathcal{L})=0
\]
if and only if $(\mathcal{X},\mathcal{L})$ is isomorphic to a product test configuration in codimension 1,
\item {\it K-stable} if 
\[
\mathrm{DF}_{\Delta}(\mathcal{X},\mathcal{L})>0
\]
for any non-almost-trivial semiample test configuration,
\item {\it uniformly K-stable} if there exists a positive constant $\epsilon>0$ such that 
\[
\mathrm{DF}_{\Delta}(\mathcal{X},\mathcal{L})\ge \epsilon I^{\mathrm{NA}}(\mathcal{X},\mathcal{L})
\]
for any semiample test configuration.
\end{itemize} 

We remark that the non-Archimedean $I$ and $J$-functionals are norm in the following sense. The non-Archimedean $I$ and $J$-functionals are nonnegative and 
\[
J^{\mathrm{NA}}(\mathcal{X},\mathcal{L})=0
\]
if and only if the normalization of $(\mathcal{X},\mathcal{L})$ is trivial by \cite[Theorem 7.9]{BHJ}. They satisfy that 
\[
\frac{1}{n}J^{\mathrm{NA}}(\mathcal{X},\mathcal{L})\le I^{\mathrm{NA}}(\mathcal{X},\mathcal{L})-J^{\mathrm{NA}}(\mathcal{X},\mathcal{L})\le nJ^{\mathrm{NA}}(\mathcal{X},\mathcal{L})
\]
by \cite[Proposition 7.8]{BHJ}.
\end{de}

\begin{rem}
Dervan \cite{De2} introduced the minimum norm of test configurations defined by $I^{\mathrm{NA}}(\mathcal{X},\mathcal{L})-J^{\mathrm{NA}}(\mathcal{X},\mathcal{L})$ and also proved that this is a norm as the non-Archimedean $I$ and $J$-functionals.
\end{rem}

The following is well-known,

\begin{prop}[{\cite[Proposition 8.2]{BHJ}}]
Let $(X,\Delta,L)$ be a polarized normal pair. Then $(X,L)$ is 
\begin{itemize}
\item K-semistable if and only if
\[
M^{\mathrm{NA}}_{\Delta}(\mathcal{X},\mathcal{L})\ge0
\]
for any normal semiample test configuration,
\item K-stable if and only if
\[
M^{\mathrm{NA}}_{\Delta}(\mathcal{X},\mathcal{L})>0
\]
for any non-trivial normal semiample test configuration,
\item uniformly K-stable if and only if there exists a positive constant $\epsilon>0$ such that 
\[
M^{\mathrm{NA}}_{\Delta}(\mathcal{X},\mathcal{L})\ge \epsilon I^{\mathrm{NA}}(\mathcal{X},\mathcal{L})
\]
for any normal semiample test configuration.
\end{itemize} 
\end{prop}

\begin{de}\label{rramamra}
For $\phi\in\mathcal{H}^\mathrm{NA}(L)$, we can define the new positive non-Archimedean metric $\phi_d$ for $d\in\mathbb{Z}_{\ge1}$ as follows. If $(\mathcal{X},\mathcal{L})$ is a representative of $\phi$, then $\phi_d$ is represented by $(\mathcal{X}^{(d)},\mathcal{L}^{(d)})$ that is the normalization of the base change $\mathcal{X}\times_{\mathbb{A}^1}\mathbb{A}^1$ via the $d$-th power map $\mathbb{A}^1\ni t\mapsto t^d\in\mathbb{A}^1$.

Note that $d\,M^{\mathrm{NA}}_{\Delta}(\phi)=M^{\mathrm{NA}}_{\Delta}(\phi_d)$ and $d\,\mathrm{DF}_{\Delta}(\phi)\ge\mathrm{DF}_{\Delta}(\phi_d)$ for any $d$. Moreover, for sufficiently divisible $d$, $M^{\mathrm{NA}}_{\Delta}(\phi_d)=\mathrm{DF}_{\Delta}(\phi_d)$ by \cite[Proposition 7.16]{BHJ}.
\end{de}
\begin{de}[Slope stability, {\cite[Definition 4.17]{RT07}}]
Let $(X,B,L)$ be a polarized log pair. For any non-void closed subscheme $Z\subset X$, {\it the deformation to the normal cone $\mathcal{X}$ of $Z$} is a test configuration for $X$ such that $\mathcal{X}$ is the blow up along $Z\times\{0\}$ (see also \cite[Chapter 5]{F}). Then $(X,B,L)$ is slope K-(semi)stable if 
\[
\mathrm{DF}_B(\mathcal{X},\mathcal{L})>0,\, (\mathrm{resp.,}\,\ge0)
\]
for any non-almost-trivial semiample test configuration such that $\mathcal{X}$ is a deformation to the normal cone of $Z$. 
\end{de}

We will use the following notation in \S \ref{DS1-stability}.
\begin{de}\label{dno}
A non-Archimedean positive metric $\phi\in\mathcal{H}^{\mathrm{NA}}(L)$ of $(X,L)$ is {\it normalized with respect to the central fiber} if we can take a representative $(\mathcal{X},\mathcal{L})$ of $\phi$ that satisfies the following,
\begin{enumerate}
 \item $(\mathcal{X},\mathcal{L})$ has a $\mathbb{G}_m$-equivariant dominant morphism $\rho: (\mathcal{X},\mathcal{L})\to (X_{\mathbb{A}^1},L_{\mathbb{A}^1})$. 
 \item The unique divisor $D=\mathcal{L}-\rho^*L_{\mathbb{A}^1}$ has the support $\mathrm{Supp}\,D\not \supset \hat{X}$ where $\hat{X}$ is the strict transformation of $X\times \{0\}$.
 \end{enumerate}
 
Whenever there is no fear of confusion, a non-Archimedean metric $\phi$ normalized with respect to the central fiber will be denoted by normalized non-Archimedean metric for short.
\end{de}
 
It follows immediately from the Definition \ref{dno} that
\begin{lem}\label{dfnf}
Let $(\mathcal{X},\mathcal{L})$ be a semiample test configuration over an $n$-dimensional polarized pair $(X,\Delta,L)$. Suppose that $(\mathcal{X},\mathcal{L})$ is normalized and there exists the birational morphism $\rho:\mathcal{X}\to X_{\mathbb{A}^1}$. Then the following hold,
\begin{itemize}
\item[(1)] $J^{\mathrm{NA}}(\mathcal{X},\mathcal{L})=-E^{\mathrm{NA}}(\mathcal{X},\mathcal{L})$.
\item[(2)] $R_\Delta^{\mathrm{NA}}(\mathcal{X},\mathcal{L})=\frac{1}{L^n}\rho^*K_{(X_{\mathbb{P}^1},\Delta_{\mathbb{P}^1})/\mathbb{P}^1}\cdot \mathcal{L}^n$.
\end{itemize}
\end{lem}

We are interested in K-stability of the following object,

\begin{de}\label{algfib}
An {\it algebraic fiber space} $f:X\to B$ is a surjective morphism between proper normal varieties with connected geometric fibers. If $F$ is the generic geometric fiber, the {\it relative dimension} $\mathrm{rel.dim}\, f$ of $f$ is $\mathrm{dim}\, F$. Note that $F$ is normal and connected.
If $f:X\to B$ is an algebraic fiber space and flat, we call $f$ a {\it flat algebraic fiber space}. If $f:X\to B$ is an algebraic fiber space, $H$ is an $f$-ample line bundle on $X$ and $L$ is an ample line bundle on $B$, then we denote $f:(X,H)\to (B,L)$ and call this a {\it polarized algebraic fiber space}. Moreover, if $\Delta$ is a boundary on $X$ and $f$ is as above, then we denote $f:(X,\Delta,H)\to (B,L)$ and call this a {\it polarized algebraic fiber space pair}.
\end{de}
\begin{de}[Adiabatic K-stability]\label{algfibk}
Let $f:X\to C$ be an algebraic fiber space, $\Delta$ be the boundary of $X$, $H$ be a $f$-ample line bundle and $L$ be an ample line bundle on $C$. Then $(X,\Delta,H)$ is 
{\it adiabatically K-(poly, semi)stable over} $(C,L)$ if $(X,\Delta,\epsilon H+f^*L)$ is K-(poly, semi)stable for sufficiently small $\epsilon>0$.
\end{de}

\begin{rem}
In the above definition, we name the above stability after the adiabatic limit technique Fine \cite{J2}, \cite{Fine} and Dervan-Sektnan \cite{DS4} used.
\end{rem}
We need the following notation about algebraic fiber spaces.
\begin{ntt}
Let $f:X\to B$ be an algebraic fiber space and $H$ and $M$ be a line bundle on $X$. $H$ and $M$ are $\mathbb{Q}$-linearly equivalent over $B$, $H\sim_{B,\mathbb{Q}}M$ if there exists a $\mathbb{Q}$-line bundle $L$ on $B$ such that $H\sim_{\mathbb{Q}}M+f^*L$. Furthermore, we will denote $H+L=H+f^*L$ if there is no fear of confusion.
\end{ntt}
For K-stability of deminormal varieties, we can calculate the Donaldson-Futaki invariant by taking the normalization. We need the following,

\begin{de}[Partially normal test configuration, {\cite[Definition 3.7]{Od}}]
Let $X$ be a deminormal scheme and $\mathcal{X}$ be a test configuration for $X$. Let also $\nu:\tilde{\mathcal{X}}\to\mathcal{X}$ be the normalization.  The {\it partially normalization} $\mathcal{X}'$ of $\mathcal{X}$ is $\mathcal{S}pec_{\mathcal{X}}(\mathcal{O}_{X\times(\mathbb{A}^1\setminus\{0\})}\cap \nu_*\mathcal{O}_{\tilde{\mathcal{X}}})$. Then, $\mathcal{X}$ is {\it partially normal} if $\mathcal{X}=\mathcal{X}'$. Note that $\mathcal{X}$ is partially normal iff $\mathcal{X}$ has no singular codimension 1 point over $0\in\mathbb{A}^1$ by \cite[Proposition 3.8]{Od}.
\end{de}

We recall the following fundamental result,

\begin{prop}[Proposition 5.1 of \cite{RT07}, Propositions 3.8 and 3.10 of \cite{Od}]
Let $(X,\Delta,L)$ be a polarized deminormal pair and $(\mathcal{X},\mathcal{L})$ be a semiample test configuration for $(X,L)$. Let also $(\tilde{\mathcal{X}},\tilde{\mathcal{L}})$ be the normalization and $(\mathcal{X}',\mathcal{L}')$ be the partially normalization. Here, $(\tilde{\mathcal{X}},\tilde{\mathcal{L}})$ is a semiample test configuration for the polarized deminormal pair $(\tilde{X},\Delta_{\tilde{X}},\nu^*L)$ where $\nu:\tilde{X}\to X$ is the normalization and $\Delta_{\tilde{X}}=\nu^{-1}_*\Delta+\mathfrak{cond}_{\tilde{X}}$. Then, 
\[
\mathrm{DF}_{\Delta_{\tilde{X}}}(\tilde{\mathcal{X}},\tilde{\mathcal{L}})=\mathrm{DF}_\Delta(\mathcal{X}',\mathcal{L}')\le\mathrm{DF}_\Delta(\mathcal{X},\mathcal{L}).
\] 
\end{prop}

\section{Review of Basic Results on K-stability and singularities}
Let us recall results on the relationship between singularities and K-stability of \cite{GIToda} first. Next, we introduce the notion of $\mathfrak{f}$-stability, which is a modified version of fibration stability introduced by Dervan and Sektnan \cite{DS4}. Then, we prove in \S5 that $\mathfrak{f}$-stability puts restrictions on the singularities on $(X,\Delta)$ as K-stability does. Moreover, we apply our results on $\mathfrak{f}$-stability to study singularities on adiabatic K-semistable log Fano fibrations.

Recall that Odaka proved that K-stability puts the restrictions on the singularities on varieties. In this section, recall those remarkable results of \cite{GIToda} and the logarithmic generalizations \cite{BHJ}.
\begin{de}
Let $Z$ be a closed subscheme of a proper normal variety $X$ and $\mathfrak{a}_Z\subset\mathcal{O}_X$ be the coherent sheaf of ideals corresponding to $Z$. Let also $\pi:\tilde{X}\to X$ be the normalization of the blow up of $X$ along $Z$. We call $\pi:\tilde{X}\to X$ {\it normalized blow up} and call the Cartier divisor $D=\pi^{-1}(Z)$ corresponding to $\pi^{-1}(\mathfrak{a}_Z)$ the {\it inverse image} of $Z$. A divisorial valuation $v$ on $X$ is {\it Rees valuation} of $Z$ if $v=\frac{\mathrm{ord}_E}{\mathrm{ord}_E(D)}$ where $E$ is an irreducible component of the exceptional divisor of $\pi$. We will denote the set of Rees valuations of $Z$ by $\mathrm{Rees}(Z)$.
\end{de}

The following is the main theorem of \cite{GIToda},
\begin{thm}[{\cite[Theorems 1.2, 1.3]{GIToda}}, {\cite[\S9]{BHJ}}]\label{ox}
Let $(X,B,L)$ be a polarized deminormal pair such that $B$ is effective. Then the following hold,
\begin{enumerate}[(i)]
\item If $(X,B,L)$ is K-semistable, then $(X,B)$ has only slc singularities.
\item If $(X,B,L)$ is K-semistable and $L=-K_X-B$, then $(X,B)$ has only klt singularities.
\end{enumerate}
\end{thm}
 We will prove the generalization of this theorem in Theorems \ref{dingular} and \ref{dFanofib}. Let us recall the proof of this theorem briefly. First, we need the following result in the minimal model program, which is called the lc modification or slc modification,
 \begin{thm}[\cite{OX}, \cite{FH}]\label{oxfh}
 \begin{enumerate}[(i)]
 \item Let $(X,B)$ be a normal log pair such that $B$ is effective. If $f:(Y,\Delta_Y)\to X$ is a birational morphism such that $K_Y+\Delta_Y$ is $f$-ample and $(Y,\lceil\Delta_Y\rceil)$ is lc with $\Delta_Y=f_*^{-1}B+\sum E$ where $E$ runs over the irreducible and reduced $f$-exceptional divisors, then we call $f:(Y,\lceil\Delta_Y\rceil)\to X$ the lc modification of $(X,B)$ and the lc modification is unique up to isomorphism.
\item  Let $(X,B)$ be a normal log pair such that $B$ is effective and $\lceil B\rceil$ is reduced. Then, there exists the lc modification $f:(Y,\Delta_Y)\to X$ of $(X,B)$.
\item Let $(X,B)$ be a deminormal log pair such that $B$ is effective and $\lceil B\rceil$ is reduced. Then, there exists a morphism $f:(Y,\Delta_Y)\to X$  such that
\begin{itemize}
\item $f$ is projective and isomorphic in codimension 1,
\item $(Y,\Delta_Y)$ is slc for $\Delta_Y=f^{-1}_*B+E$ where $E$ is the sum of all reduced $f$-exceptional divisors,
\item $K_Y+\Delta_Y$ is $f$-ample.
\end{itemize} We call such $f:(Y,\Delta_Y)\to X$ the slc modification of $(X,B)$ and the slc modification is unique up to isomorphism.
\end{enumerate}
 \end{thm}
 
 Let us recall the proof briefly. To prove (ii), we apply the results of \cite{HX} and \cite{Fuj} (see the detail in \cite{OX}). (i) follows from the argument of \cite[\S 4]{FH}. Here, note that we do not assume that $\lceil B\rceil$ is reduced. To prove (iii), we take the normalization $\nu:(X',B'+D)\to (X,B)$ where $D$ is the conductor divisor on $X'$. Then, there exists the lc modification $f':(Y',\Delta_Y')\to (X',B'+D)$ by (ii). Let $n:D^{\mathrm{n}}\to D$ be the normalization and $D'=f'^{-1}_*D$. In general, $\lceil\mathrm{Diff}_{D}(B')\rceil$ is not reduced. Here, $\mathrm{Diff}_{D}(B')$ is the different and see the definition in \cite[\S16]{K+}. However, if $n':D'^{\mathrm{n}}\to D'$ is the normalization,
\[
(D'^{\mathrm{n}},\mathrm{Diff}_{D'}(\Delta_Y-D'))\to (D^{\mathrm{n}},\mathrm{Diff}_{D}(B'))
\]
is the lc modification in the sense of (i) (cf., \cite[4.4]{FH}). Therefore, the involution of $D^{\mathrm{n}}$ lifts to $D'^{\mathrm{n}}$ and (iii) follows from Koll\'{a}r's gluing theorem \cite[Theorem 5.13]{Ko} (see the detail in \cite{OX} and \cite{FH}).

By Theorem \ref{oxfh}, we can conclude as follows. 
\begin{cor}
Let $(X,B)$ be a non-slc pair such that $B$ is effective and $\lceil B\rceil$ is reduced. Let also $\nu:(X',\nu_*^{-1}B+\mathfrak{cond})\to (X,B)$ be the normalization. Then, there exists a closed subscheme $Z$ on $X$ such that $A_{(X,B)}(v)\coloneq A_{(X',\nu_*^{-1}B+\mathfrak{cond})}(v)<0$ for $v\in\mathrm{Rees}(\nu^{-1}Z)$.
\end{cor}

We also remark the following theorem,
\begin{thm}[{\cite[Theorem 4.8]{BHJ}}]\label{bhj48}
Let $Z\subset X$ be a closed subscheme of a normal variety. Then, if $\mathcal{X}$ is the normalization of the deformation to the normal cone of $Z$, the set of Rees valuations of $Z$ coincides with the set of valuations $v_E$, where $E$ runs over the irreducible components of $\mathcal{X}_0$.
\end{thm}

Therefore, Theorem \ref{ox} follows from the calculation as \cite[Proposition 9.12]{BHJ}. We will prove in \S\ref{Dsing} that $\mathfrak{f}$-stability (we will define in \S \ref{DS1-stability} below) puts restrictions on the singularity as the above argument.
\section{Fibration stability after Dervan-Sektnan}\label{DS1-stability}
 In this section, we introduce the notion of $\mathfrak{f}$-stability that is a modified version of fibration stability, which is introduced by Dervan and Sektnan \cite{DS4}.
 Next, we show the fundamental results of $\mathfrak{f}$-stability (e.g., Theorem \ref{cc}) and compare $\mathfrak{f}$-stability with fibration stability.
 
First, recall the notion of fibration degeneration introduced by Dervan and Sektnan \cite{DS4} as test configuration in K-stability.
\begin{de}\label{dedede}
Let $\pi:(X,H)\to (B,L)$ be a polarized algebraic fiber space (cf., Definition \ref{algfib}). $\Pi:(\mathcal{X},\mathcal{H})\to B\times\mathbb{A}^1$ is a {\it fibration degeneration} for $\pi$ if the following hold,
\begin{itemize}
\item $(\mathcal{X},\mathcal{H}+j\Pi^*(L\otimes\mathcal{O}_{\mathbb{A}^1}))$ is an ample test configuration over $(X,H+j\pi^*L)$ for sufficiently large $j>0$,
\item $\Pi$ is $\mathbb{G}_m$-equivariant over $B\times\mathbb{A}^1$,
\item Let $\Pi_1$ be the restriction of $\Pi$ to the fiber over $1\in\mathbb{A}^1$. Then $\pi=\Pi_1$.
\end{itemize}
 
We will need a weaker concept which we call a {\it semi fibration degeneration} where $\mathcal{H}$ is relatively semiample over $B\times\mathbb{A}^1$.
\end{de}
 
This notion coincides with \cite[Definition 2.16]{DS4} due to [{\it loc.cit}, Lemma 2.15 and Corollary 2.24] and next proposition when $f$ is flat.
 
\begin{prop}\label{ds}
Let $\pi:(X,H)\to (B,L)$ be a polarized algebraic fiber space. If $(\mathcal{X},\mathcal{H})$ is a test configuration such that $(\mathcal{X},\mathcal{H}+j\pi^*L_{\mathbb{A}^1})$ is a semiample test configuration for sufficiently large $j$, where $\mathcal{X}$ dominates the trivial test configuration $X\times\mathbb{A}^1$ (hence also $B\times\mathbb{A}^1$), then there exists a {\it fibration degeneration} for $\pi$ dominated by $(\mathcal{X},\mathcal{H})$.
\end{prop}
\begin{proof}
By assumption, $(\mathcal{X},\mathcal{H}+j\pi^*L_{\mathbb{A}^1})$ is semiample for some large $j$ and hence $\mathcal{H}$ is semiample over $B\times\mathbb{A}^1$. Therefore, if $\Pi$ is a canonical morphism from $\mathcal{X}$ to $B\times\mathbb{A}^1$, $\Pi^*\Pi_*(\mathcal{H}^k)\twoheadrightarrow \mathcal{H}^k$ induces a morphism $f_k:\mathcal{X}\to \mathcal{P}roj_{B\times\mathbb{A}^1}(\bigoplus_{l\ge0}\mathcal{S}ym^l\Pi_*(\mathcal{H}^{k}))$ for sufficiently large $k$. If $(\mathcal{Y}_k,\frac{1}{k}\mathcal{O}_{\mathcal{P}roj}(1))$ is the image of $f_k$, this is a fibration degeneration for $\pi$ dominated by $(\mathcal{X},\mathcal{H})$.
\end{proof}

Then we define the almost triviality of fibration degenerations similarly to test cinfigurations.

\begin{de}
A semi fibration degeneration $(\mathcal{X},\mathcal{H})$ is {\it almost trivial} if $(\mathcal{X},\mathcal{H})$ is almost trivial as a test configuration.
\end{de}

Now, we can define $\mathfrak{f}$-stability as follows,

\begin{de}\label{ds1}
Suppose that $f:(X,\Delta,H)\to (B,L)$ is a polarized algebraic fiber space with a boundary $\Delta$. Let $m=\mathrm{rel.dim}\, f$ and $\mathrm{dim}\, B=n$. For any normal fibration degeneration $(\mathcal{X},\mathcal{H})$ for $f$, we define constants $W^{\Delta}_0(\mathcal{X},\mathcal{H}),\, W^{\Delta}_1(\mathcal{X},\mathcal{H}), \cdots ,W^{\Delta}_n(\mathcal{X},\mathcal{H})$ and a rational function $W^{\Delta}_{n+1}(\mathcal{X},\mathcal{H})(j)$ so that the partial fraction decomposition of $M^\mathrm{NA}_{\Delta}(\mathcal{X},\mathcal{H}+jL)$ in $j$ is as follows:
\[
V(H+jL)M^\mathrm{NA}_{\Delta}(\mathcal{X},\mathcal{H}+jL)=W^{\Delta}_{n+1}(\mathcal{X},\mathcal{H})(j) + \sum_{i=0}^n j^iW^{\Delta}_{n-i}(\mathcal{X},\mathcal{H}).
\]
 See also \cite[Lemma 2.25]{DS4}. Then $f:(X,\Delta,H)\to (B,L)$ is called
\begin{itemize}
\item {\it $\mathfrak{f}$-semistable} if $W^{\Delta}_0\ge 0$ and $$W^{\Delta}_0(\mathcal{X},\mathcal{H})=W^{\Delta}_1(\mathcal{X},\mathcal{H})=\cdots =W^{\Delta}_i(\mathcal{X},\mathcal{H})=0\Rightarrow W^{\Delta}_{i+1}(\mathcal{X},\mathcal{H})\ge0$$ for $i=0,1,\cdots ,n-1$ for any fibration degeneration. Equivalently, $f$ is $\mathfrak{f}$-semistable if $\sum_{i=0}^n j^iW^{\Delta}_{n-i}(\mathcal{X},\mathcal{H})\ge 0$ for sufficiently large $j>0$.
\item {\it $\mathfrak{f}$-stable} if $f$ is $\mathfrak{f}$-semistable and $$W^{\Delta}_i(\mathcal{X},\mathcal{H})= 0,\, i=0,1,\cdots ,n-1\, \Rightarrow W^{\Delta}_n(\mathcal{X},\mathcal{H})>0$$ for any non-almost-trivial fibration degeneration of $(X,H)$.
\end{itemize}
Note that $f$ is
\begin{itemize}
\item $\mathfrak{f}$-semistable if and only if $W^{\Delta}_0\ge 0$ and $W^{\Delta}_0(\mathcal{X},\mathcal{H})=W^{\Delta}_1(\mathcal{X},\mathcal{H})=\cdots =W^{\Delta}_i(\mathcal{X},\mathcal{H})=0\Rightarrow W^{\Delta}_{i+1}(\mathcal{X},\mathcal{H})\ge0$ for $i=0,1,\cdots ,n-1$ for any test configuration $(\mathcal{X},\mathcal{H})$ dominating $X_{\mathbb{A}^1}$ such that $(\mathcal{X},\mathcal{H}+jL)$ is a semiample test configuration for sufficiently large $j$.
\item $\mathfrak{f}$-stable if and only if $f$ is $\mathfrak{f}$-semistable and $W^{\Delta}_i(\mathcal{X},\mathcal{H})=0,\, i=0,1,\cdots ,n-1\, \Rightarrow W^{\Delta}_n(\mathcal{X},\mathcal{H})>0$ for any non almost trivial test configuration $(\mathcal{X},\mathcal{H})$ for $(X,H)$ dominating $X_{\mathbb{A}^1}$ such that $\mathcal{H}+jL$ is semiample for sufficiently large $j$.
\end{itemize}
This follows immediately from Proposition \ref{ds}. If $\Delta=0$, we will denote $W^{\Delta}_i=W_i$.
\end{de}
 
We remark the following trivial lemma.
 
\begin{lem}\label{trlem}
Let $f:(X,\Delta,H)\to (B,L)$ be a polarized algebraic fiber space pair. Then $f:(X,\Delta,H)\to (B,L)$ is $\mathfrak{f}$-(semi)stable if and only if so is $f:(X,\Delta,H+jL)\to (B,L)$ for $j\in\mathbb{Q}$.
\end{lem}
\begin{rem}\label{trrem}
In the original definition of $W_i$ in \cite{DS}, Dervan-Sektnan used $\mathrm{DF}$ instead of $M^\mathrm{NA}$. In other words, we can define $W_i'^{\Delta}$ for any semi fibration degeneration $(\mathcal{X},\mathcal{H})$ for $f$ as 
\[
V(H+jL)\mathrm{DF}_{\Delta}(\mathcal{X},\mathcal{H}+jL)=W'^{\Delta}_{n+1}(\mathcal{X},\mathcal{H})(j) + \sum_{i=0}^n j^iW'^{\Delta}_{n-i}(\mathcal{X},\mathcal{H}).
\]
 However, $\mathfrak{f}$-stability defined by using $\mathrm{DF}$ coincides with the above one i.e., the positivity of $\sum_{i=0}^n j^iW^{\Delta}_{n-i}(\mathcal{X},\mathcal{H})$ coincides with $\sum_{i=0}^n j^iW'^{\Delta}_{n-i}(\mathcal{X},\mathcal{H})$. In fact, it follows from $$V(H+jL)(\mathrm{DF}(\mathcal{X},(\mathcal{H}+jL))-M^\mathrm{NA}(\mathcal{X},(\mathcal{H}+jL)))=((\mathcal{X}_0-\mathcal{X}_{0,\mathrm{red}})\cdot (\mathcal{H}+jL)^{\mathrm{dim}\, X})\ge 0$$ and Proposition 7.16 of \cite{BHJ}. On the other hand, we can easiliy see that if $H$ is ample and $(X,\Delta,H)$ is adiabatic K-semistable with respect to $L$, then $f$ is $\mathfrak{f}$-semistable.
 
We also remark that $\mathfrak{f}$-stability is different from the notion of fibration stability of \cite{DS4} in that fibration stability is only the positivity of $W_0$ and $W_1$. See also Remark \ref{fibstfstdsst}. $\mathfrak{f}$-stability coincides with the previous version \cite{DS} of fibration stability when $B$ is a curve and all fibers are K-polystable.
\end{rem}
Next, we want to calculate $W_0$ and $W_1$. For simplicity, we need the following notation,
\begin{ntt}
  If $f:X\to B$ is a morphism of varieties and $D$ is a Cartier divisor on $B$ such that $f(X)\not\subset D$, then we can define the Cartier divisor $f^*D$ and we denote it by $X\cap D$.
  \end{ntt}

\begin{prop}\label{cal}
Let $\pi:(X,H)\to (B,L)$ be a polarized algebraic fiber space. Suppose that $m=\mathrm{rel.dim}\, \pi$ and $n=\mathrm{dim}\,B \ge 2$. Let $(\mathcal{X},\mathcal{H})$ be a normal semi fibration degeneration for $\pi$. Then, for sufficiently large $k\in\mathbb{Z}$, there is $D\in |kL|$ such that $(\mathcal{X}\cap D,\mathcal{H}_{|\mathcal{X}\cap D})$ is the test configuration of $(X\cap D,H_{|X\cap D})$ and $k\binom{n+m}{n-1}W_1(\mathcal{X},\mathcal{H})=\binom{n+m-1}{n-2}W_1(\mathcal{X}\cap D,\mathcal{H}_{|\mathcal{X}\cap D})$.
\end{prop}
\begin{proof}
As in \cite[p. 17]{DS4}, let
\begin{align*}
C_1(\mathcal{X},\mathcal{H})&=\frac{m}{m+2}\left(\frac{-K_{X/B}\cdot L^n\cdot H^{m-1}}{L^n\cdot H^m}\right)L^{n-1}\cdot\mathcal{H}^{m+2},\\
C_2(\mathcal{X},\mathcal{H})&=-\frac{m}{m+1}\left(\frac{(-K_{X/B}\cdot L^n\cdot H^{m-1})(L^{n-1}\cdot H^{m+1})}{(L^n\cdot H^m)^2}\right)L^{n}\cdot\mathcal{H}^{m+1},\\
C_3(\mathcal{X},\mathcal{H})&=\left(\frac{-K_{X}\cdot L^{n-1}\cdot H^{m}}{L^n\cdot H^m}\right)L^{n}\cdot\mathcal{H}^{m+1},\\
C_4(\mathcal{X},\mathcal{H})&=L^{n-1}\cdot\mathcal{H}^{m+1}\cdot K^{\mathrm{log}}_{\mathcal{X}/\mathbb{P}^1}.
\end{align*}
Then, we obtain
\[
W_1(\mathcal{X},\mathcal{H})=\binom{n+m}{n-1}\sum_{j=1}^4C_j(\mathcal{X},\mathcal{H}).
\]
The set of $D\in |kL|$ such that $X\cap D,\mathcal{X}\cap D$ and $D$ are normal is Zariski-dense open subset of $|kL|$ for $k\gg 0$ by the Bertini theorem. We may assume that $k=1$ by replacing $L$. Moreover, $X\cap D,\mathcal{X}\cap D,D$ are also connected by \cite[III Exercise 11.3]{Ha}. Furthermore, since $\mathcal{X}\cap D$ dominates $\mathbb{P}^1$, it is flat over $\mathbb{P}^1$. Thus, $(\mathcal{X}\cap D,\mathcal{H}_{|\mathcal{X}\cap D})$ is an ample test configuration over $X\cap D$. We can apply the adjunction formula in \cite[\S 16, \S 17]{K+} for $D$ and the pullbacks of this to $X$ and $\mathcal{X}$ since the pullbacks of $D$ are normal Cartier divisors. By these facts, we obatin
\begin{align*}
(K_X-\pi^*K_B)_{|\pi^*D}&=K_{\pi^*D}-\pi^*K_D+(\pi^*D-\pi^*D)_{|\pi^*D} \\
&=K_{\pi^*D}-\pi^*K_D.
\end{align*}
Note that $K_{\mathcal{X}}$ and $\mathcal{X}_{0,\mathrm{red}}$ are Cartier in codimension 1 but not $\mathbb{Q}$-Cartier entirely in general. However, we can choose $D$ so general that $K_{\mathcal{X}}$ is also Cartier on any point of codimension 1 in $D$ and then the adjunction formula holds for such $D$ (cf., \cite[16.4.3]{K+}). Hence,  we obtain $C_1(\mathcal{X},\mathcal{H})+C_2(\mathcal{X},\mathcal{H})=C_1(\mathcal{X}\cap D,\mathcal{H}_{|\mathcal{X}\cap D})+C_2(\mathcal{X}\cap D,\mathcal{H}_{|\mathcal{X}\cap D})$. Let $\Pi : \mathcal{X}\to B$ the canonical projection. Since $\mathcal{X}\cap D=\Pi^*D$ and $X\cap D=\pi^*D$, we also have
\begin{align*}
C_3(\mathcal{X},\mathcal{H})&+C_4(\mathcal{X},\mathcal{H})=\left(\frac{-K_X\cdot L^{n-1}H^m}{L^nH^m}\right) L^n\cdot \mathcal{H}^{m+1}+L^{n-1}\cdot \mathcal{H}^{m+1}\cdot K^{\mathrm{log}}_{\mathcal{X}/\mathbb{P}^1} \\
&= \left(\frac{-(K_{\pi^*D}-L_{|\pi^*D})\cdot L_{|\pi^*D}^{n-2}H_{|\pi^*D}^m}{L_{|\pi^*D}^{n-1}H_{|\pi^*D}^m}\right) L_{|\Pi^*D}^{n-1}\cdot \mathcal{H}_{|\Pi^*D}^{m+1}+L_{|\Pi^*D}^{n-2}\cdot \mathcal{H}_{|\Pi^*D}^{m+1}\cdot (K^{\mathrm{log}}_{{\Pi ^*D}/\mathbb{P}^1}-L_{|\Pi^*D}) \\
&= \left(\frac{-K_{\pi^*D}\cdot L_{|\pi^*D}^{n-2}H_{|\pi^*D}^m}{L_{|\pi^*D}^{n-1}H_{|\pi^*D}^m}\right) L_{|\Pi^*D}^{n-1}\cdot \mathcal{H}_{|\Pi^*D}^{m+1}+L_{|\Pi^*D}^{n-2}\cdot \mathcal{H}_{|\Pi^*D}^{m+1}\cdot K^{\mathrm{log}}_{{\Pi ^*D}/\mathbb{P}^1},
\end{align*}
and hence we have $C_3(\mathcal{X},\mathcal{H})+C_4(\mathcal{X},\mathcal{H})=C_3(\mathcal{X}\cap D,\mathcal{H}_{|\mathcal{X}\cap D})+C_4(\mathcal{X}\cap D,\mathcal{H}_{|\mathcal{X}\cap D})$.
\end{proof}
 Now, we can calculate $W_0$ and $W_1$ directly as follows. First, we need the following lemma,
 \begin{lem}[Lemma 2.33 of \cite{DS4}]
 Let $(\mathcal{X},\mathcal{H})$ be a normal semi fibration degeneration for a polarized algebraic fiber space $f:(X,H)\to (B,L)$. Then
\[
W_0(\mathcal{X},\mathcal{H})=\binom{n+m}{n}V(L)\cdot M^\mathrm{NA}(\mathcal{X}_b,\mathcal{H}_b)
\]
holds for general $b\in B$. 
\end{lem}
 
\begin{proof}
For the readers' convenience, we prove this lemma here. First, we prove the following claim. Let $\mathcal{X}$ be a fibration degeneration for $f:X\to B$. Then a general fiber of $\mathcal{X}$ over $B$ is a test configuration $\mathcal{X}_b$ for $X_b$.

It suffices to show that $\mathcal{X}_b$ is flat over $\mathbb{A}^1$. Since $B$ is normal, $\mathcal{X}$ is flat over codimension 1 points of $B\times \mathbb{A}^1$. Therefore, $\mathcal{X}$ is also flat over $(b,0)$ for general points $b\in B$. On the other hand, by cutting by ample divisors, we may assume that $\mathrm{dim}\, B=1$ and may also assume that $X$ is flat over $B$ due to the same argument of the proof of Proposition \ref{cal} and the Bertini theorem. Therefore, $X\times(\mathbb{A}^1\setminus\{0\})$ is flat over $B\times(\mathbb{A}^1\setminus\{0\})$. Hence, for general point $b\in B$, $\mathcal{X}_b$ is flat over $\mathbb{A}^1$. Thus, the assertion of the lemma follows from as in the proof of Proposition \ref{cal}.
\end{proof}
 
To calculate $W_1$, it suffices to show the following by Proposition \ref{cal}:
\begin{lem}\label{cal2}
Let $\pi:(X,H)\to (B,L)$ be a polarized algebraic fiber space over a smooth curve $B$. Let $\mathrm{deg}\, L=1$. If $(\mathcal{X},\mathcal{H})$ is a normal semi fibration degeneration for $\pi$, then
\[
W_1(\mathcal{X},\mathcal{H})=V(H)(M^\mathrm{NA}(\mathcal{X},\mathcal{H})+(S(X_b,H_b)-S(X,H))(E^{\mathrm{NA}}(\mathcal{H})-E^{\mathrm{NA}}(\mathcal{H}_b))).
\]
\end{lem}
\begin{proof}
Notations as in the proof of Proposition \ref{cal}. Then, it follows that 
\begin{align*}
C_1(\mathcal{X},\mathcal{H})&=V(H)E^{\mathrm{NA}}(\mathcal{H})S(X_b,H_b) \\
C_2(\mathcal{X},\mathcal{H})&=-V(H)E^{\mathrm{NA}}(\mathcal{H}_b)S(X_b,H_b) \\
C_3(\mathcal{X},\mathcal{H})&=V(H)E^{\mathrm{NA}}(\mathcal{H}_b)S(X,H) \\
C_4(\mathcal{X},\mathcal{H})&=V(H)(M^\mathrm{NA}(\mathcal{X},\mathcal{H})-E^{\mathrm{NA}}(\mathcal{H})S(X,H))
\end{align*}
immediately from a direct computation.
\end{proof}

For general cases, we can calculate $W_1$ by cutting $B$ by general divisors in $|L|_{\mathbb{Q}}$.
 
Therefore, we can conclude that the original fibration stability defined by Dervan and Sektnan \cite{DS} is the property of codimension 1 points of the base $B$ as we see in the next remark.
 
\begin{rem}\label{fibstfstdsst}
In \cite[Definitions 2.26, 2.28]{DS4}, the fibration stability is defined as the positivity of $W_0(\mathcal{X},\mathcal{H})$ and $W_1(\mathcal{X},\mathcal{H})$ for fibration degeneration $(\mathcal{X},\mathcal{H})$ whose minimum norm 
\[
\|(\mathcal{X},\mathcal{H})\|_m=\frac{L^n\cdot \mathcal{H}^{m+1}}{m+1}+L^n\cdot \mathcal{H}^m\cdot (\mathcal{H}-H_{\mathbb{P}^1})
\] (cf., \cite[Definition 2.27]{DS4}) does not vanish. Here, the notations as in Proposition \ref{cal}. It is easy to see that a normal fibration degeneration that has the minimum norm $\|(\mathcal{X},\mathcal{H})\|_m=0$ is not necessarily trivial but the general fiber is the trivial test configuration (see \cite[Remark 2.30]{DS4}). Indeed, let $f:X\to B$ be an algebraic fiber space over a surface and $(\mathcal{X},\mathcal{H})$ be a deformation to the normal cone of a point $x\in X$. Then $\mathcal{X}$ has a general fiber that is a trivial test configuration and $W_0(\mathcal{X},\mathcal{H})=W_1(\mathcal{X},\mathcal{H})=0$ by Proposition \ref{cal}. In this respect, $\mathfrak{f}$-stability is different from fibration stability even when the base is a curve. See also Example \ref{lcbaseex}.
 
On the other hand, we obtain the following when $\mathcal{X}$ is flat over $B\times\mathbb{A}^1$:
\begin{prop}\label{pr}
Let $\pi:(X,H)\to (B,L)$ be a flat polarized algebraic fiber space. Let $(\mathcal{X},\mathcal{H})$ is a normal fibration degeneration for $\pi$ such that $(\mathcal{X}\cap D,\mathcal{H}_{|\mathcal{X}\cap D})$ is a trivial test configuration for general $D\in |mL|$ and for sufficiently large $m\gg 0$. If the canonical morphism $\Pi:\mathcal{X}\to B\times\mathbb{A}^1$ is flat, then $(\mathcal{X},\mathcal{H})$ is the trivial test configuration.
\end{prop}
 
\begin{proof}
We may assume that $\mathcal{X}\cap D$ is trivial for general member $D$ of $|mL|,\, m\gg 0$. By assumption, $(\mathcal{X},\mathcal{H}+jL)$ is an ample test configuration for all $j \gg 0$, and hence it suffices to show that $\mathcal{X}$ and $X\times \mathbb{A}^1$ is isomorphic in codimension $1$ by the fact that if $(\mathcal{X},\mathcal{H}+jL)$ is an almost trivial test configuration and normal then it is trivial. Let $a_i\in \mathbb{N},\, D_i$ be the irreducible divisors and $\mathcal{X}_0=\sum a_iD_i$ be the central fiber over $\mathbb{A}^1$. We can check that $p_0:\mathcal{X}_0 \to B$ is a flat proper fibration and hence all the $D_i$ dominates $B$ via $p_0$. For general $D$, $\mathcal{X}_0\cap D=\sum a_i(D_i\cap D)$ is the cycle decomposition and coincides to the integral cycle $X\cap D$. Therefore, $\mathcal{X}_0$ is integral. 

By the valuative criterion for properness \cite[II Theorem 4.7]{Ha}, we get a rational map $\psi :\mathcal{X}_0 \dashrightarrow X$ from $\phi :\mathcal{X} \dashrightarrow X\times \mathbb{A}^1$ the canonical birational map. Let $\Gamma$ be the normalized graph of $\phi$ and $p:\Gamma \to \mathcal{X}\, ,q:\Gamma \to X$ be the canonical projection respectively. Let $\Gamma_0$ be the central fiber of $\Gamma$ over $0\in\mathbb{A}^1$ and $E_0$ be the irreducible component of $\Gamma_0$ corresponding to $\mathcal{X}_0$. We must show that $q_{|E_0}$ is birational. Take the largest open subset $U\subset \mathcal{X}$ such that $U$ has a morphism $\iota:U\to \Gamma$ that is the local inverse of $p$. Let $U_0$ be the central fiber of $U$ over $\mathbb{A}^1$. Note that $U_0\ne\emptyset$ since $\mathrm{codim}(\mathcal{X}\setminus U)\ge2$. 

Take $D\in |mL|$ such that $\mathcal{X}\cap D$, $\mathcal{X}_0\cap D$, $X\cap D$, $E_0\cap D$ and $\Gamma\cap D$ are normal by the theorem of Bertini. We may assume that $\mathcal{X}\cap D$, $\mathcal{X}_0\cap D$, $X\cap D$, $E_0\cap D$ and $\Gamma\cap D$ are also connected. Indeed, if $\mathrm{dim}\, B\ge 2$, they are connected. Otherwise, $\mathrm{dim}\, B=1$. Then, we may assume that $\mathrm{deg}\, D=1$ and take a general fiber over $B$. We remark that $E_0\cap D\ne\emptyset$ because $E_0$ dominates $B$. Since $\mathcal{X}\cap D$ and $X_{\mathbb{A}^1|D}$ are isomorphic, $U\cap D$ and $\Gamma\cap D$ are birational equivalent. Assume that $q_{|E_0}:E_0\to X$ is not a birational morphism. Assume also that $q_{|E_0}$ is not dominant, then $q_{|E_0\cap D}$ is not birational for general $D$. This contradicts to that $U_0\cap D$ is embedded into $X$ via $q\circ \iota$. Hence $q_{|E_0}$ is dominant. Since we assume that $q_{|E_0}$ is not birational, $\mathrm{deg}(q_{|E_0})>1$. This contradicts to $\mathrm{deg}(q_{|E_0})=\mathrm{deg}(q_{|E_0\cap D})=1$. Therefore, $q_{|E_0}:E_0\to X\times\{0\}$ is a birational morphism. 
\end{proof}
\end{rem}

To prove Theorem \ref{singular} below, we need to calculate $W_{i}^{\Delta}$ of an arbitrary normalized semi fibration degeneration $(\mathcal{X},\mathcal{H})$ for $\pi$ and $i\ge 2$ as in Proposition \ref{cal}. Indeed, we can prove the logarithmic version of Proposition \ref{cal} similarly. However, the coefficients are more complicated when $i\ge2$ and we want to ignore the effects of higher codimensional points on $B$.

Let $f:(X,\Delta,H)\to (B,L)$ be a polarized algebraic fiber space pair. Suppose that $m=\mathrm{rel.dim}\, f$ and $n=\mathrm{dim}\,B$. First, we may assume that $H$ is ample by Lemma \ref{trlem}. Take general ample divisors $D_1,D_2,\cdots ,D_n\in |ML|$ for a fixed integer $M\gg0$ in the sense of the Bertini theorem. Let $f(j)$ and $g(j)$ be $-K_{(X,\Delta)}\cdot (H+jL)^{n+m-1}$ and $ (H+jL)^{n+m}$ respectively. Note that $f(j)$ and $g(j)$ depend only on $L$ and $H$. On the other hand, $$S(X,\Delta,H+jL)=\frac{f(j)}{g(j)}.$$ Note also that 
\[
(\mathcal{H}+jL)^{n+m+1}=\sum_{i=0}^n j^{n-i}\binom{n+m+1}{n-i}\mathcal{H}^{m+1+i}\cdot L^{n-i}.
\]

 Then we can calculate $W^{\Delta}_i$ as follows:
\begin{lem}\label{induction}
Let $f:(X,\Delta,H)\to (B,L)$ be a polarized algebraic fiber space pair with an ample line bundle $H$ and $(\mathcal{X},\mathcal{H})$ be a normal semiample test configuration for $(X,H)$ dominating $X_{\mathbb{A}^1}$. Suppose also that $m=\mathrm{rel.dim}\, f$, $n=\mathrm{dim}\,B$ and $(\mathcal{X},\mathcal{H})$ is normalized with respect to the central fiber. Then, there exist constants $C_{n-i+1},\cdots ,C_{n}$ independent from $(\mathcal{X},\mathcal{H})$ and $j$ such that
\begin{align*}
W_{i}^{\Delta}(\mathcal{X},\mathcal{H}) &= \binom{n+m}{n-i} \left(K^\mathrm{log}_{(\mathcal{X},\mathcal{D})/\mathbb{P}^1}\cdot L^{n-i} \cdot \mathcal{H}^{m+i}+\frac{S(X_b,\Delta_b,H_b)}{m+i+1}\mathcal{H}^{m+i+1}\cdot L^{n-i}\right)\\
&+\sum_{j=n-i+1}^n C_jJ^{\mathrm{NA}}(\mathcal{H}_{|\mathcal{X}\cap D_1 \cap D_2 \cap \cdots \cap D_{j}}),
\end{align*}
where $D_1,D_2,\cdots ,D_n\in|ML|$ are general ample divisors for sufficiently large integer $M>0$ that $ML$ is very ample. Here, $\mathcal{D}$ is the strict transformation of $\Delta\times\mathbb{P}^1$ on $\mathcal{X}$. Furthermore, 
\begin{align*}
W_{i} ^{\Delta}(\mathcal{X},\mathcal{H})&= M^{-(n-i)}\binom{n+m}{n-i} \Biggl(K^\mathrm{log}_{(\mathcal{X}\cap D_1 \cap D_2 \cap \cdots \cap D_{n-i},\mathcal{D}\cap D_1 \cap D_2 \cap \cdots \cap D_{n-i})/\mathbb{P}^1} \cdot \mathcal{H}_{|\mathcal{X}\cap D_1 \cap D_2 \cap \cdots \cap D_{n-i}}^{m+i}\\
&+\frac{S(X_b,\Delta_b,H_b)}{m+i+1}\mathcal{H}_{|\mathcal{X}\cap D_1 \cap D_2 \cap \cdots \cap D_{n-i}}^{m+i+1} \Biggr)+\sum_{j=n-i+1}^n C_jJ^{\mathrm{NA}}(\mathcal{H}_{|\mathcal{X}\cap D_1 \cap D_2 \cap \cdots \cap D_{j}}).
\end{align*}
\end{lem}
\begin{ntt}
If there is no fear of confusion, we will omit $\mathcal{X}\cap D_1 \cap D_2 \cap \cdots \cap D_{j}$ and denote as
\[
J^{\mathrm{NA}}(\mathcal{H}_{|\mathcal{X}\cap D_1 \cap D_2 \cap \cdots \cap D_{j}})=J^{\mathrm{NA}}(\mathcal{X}\cap D_1 \cap D_2 \cap \cdots \cap D_{j},\mathcal{H}_{|\mathcal{X}\cap D_1 \cap D_2 \cap \cdots \cap D_{j}}).
\]
\end{ntt}

\begin{proof}[Proof of Lemma \ref{induction}]
Note that $(\mathcal{X}\cap D_1,\mathcal{H}|_{\mathcal{X}\cap D_1})$ is also normalized with respect to the central fiber (cf., Definition \ref{dno}). Replacing $L$ by $ML$, we may assume that $M=1$.
For the first assertion,
\begin{align*}
S(X,\Delta,H+jL)=&-(n+m)\frac{K_{(X,\Delta)}\cdot (H+jL)^{n+m-1}}{(H+jL)^{n+m}}\\
=&S(X_b,\Delta_b,H_b)+O(j^{-1})
\end{align*}
where the coefficients of $O(j^{-1})$ depend only on $n,m,H$ and $L$. On the other hand, the $j^{n-i}$-term of $(\mathcal{H}+jL)^{n+m+1}$ is as follows
\[
\binom{n+m}{n-i}\frac{m+i+1}{n+m+1}\mathcal{H}^{m+i+1}\cdot L^{n-i},
\]
and the higher degree terms of $(\mathcal{H}+jL)^{n+m+1}$ are the forms of $C_{n-i+l}J^{\mathrm{NA}}(\mathcal{H}_{|\mathcal{X}\cap D_1 \cap D_2 \cap \cdots \cap D_{n-i+l}})$ where $l>0$ and $C_{n-i+l}$ is a constant depending only on $l$ by the argument before this lemma. Here, we have $$J^{\mathrm{NA}}(\mathcal{H}_{|\mathcal{X}\cap D_1 \cap D_2 \cap \cdots \cap D_{k}})=-E^{\mathrm{NA}}(\mathcal{H}_{|\mathcal{X}\cap D_1 \cap D_2 \cap \cdots \cap D_{k}})$$ for $k$ by Lemma \ref{dfnf}. Thus, the first assertion holds. We can take $D_i\in |L|$ so general that $K_{\mathcal{X}\cap D_1\cap\cdots\cap D_{i-1}}$, $\mathcal{X}_{0,\mathrm{red}}\cap D_1\cap\cdots\cap D_{i-1}$ and $\mathcal{D}\cap D_1\cap\cdots\cap D_{i-1}$ are Cartier on any codimension 1 point of $\mathcal{X}\cap D_1\cap\cdots\cap D_i$. By the adjunction formula \cite[\S16]{K+}, for example, we have
\[
K^\mathrm{log}_{(\mathcal{X},\mathcal{D})/\mathbb{P}^1}|_{\mathcal{X}\cap D_1}=K^\mathrm{log}_{(\mathcal{X}\cap D_1,\mathcal{D}\cap D_1)/\mathbb{P}^1}-\mathcal{X}\cap D_1|_{\mathcal{X}\cap D_1},
\]
as in the proof of Proposition \ref{cal} and hence we obtain the second assertion. 
\end{proof}

On the other hand, we calculate $W_i$ as Lemma \ref{induction} when there exists a line bundle $L_0$ on $B$ such that $K_{(X,\Delta)}\equiv\lambda\,H+f^*L_0$ for $\lambda\in\mathbb{Q}$.
\begin{prop}\label{me}
Let $f:(X,\Delta,H)\to (B,L)$ be a polarized algebraic fiber space pair with $m=\mathrm{rel.dim}\, f$ and $n=\mathrm{dim}\,B$. Furthermore, suppose that $H$ is ample and $L$ is very ample. Let $(\mathcal{X},\mathcal{H})$ be an arbitrary normal semiample test configuration for $(X,H)$ normalized with respect to the central fiber. Suppose also that $(\mathcal{X},\mathcal{H})$ dominates $X_{\mathbb{A}^1}$ and that there exist $\lambda \in \mathbb{Q}$ and a line bundle $L_0$ on $B$ such that $\lambda H\equiv K_X+\Delta+f^*L_0$. Then
\begin{align}
\left(\binom{n+m}{n-k}(H^{m+k}\cdot L^{n-k})\right)^{-1}&W^{\Delta}_k(\mathcal{X},\mathcal{H})= H_{\Delta\cap D_1 \cap D_2 \cap \cdots \cap D_{n-k}}^{\mathrm{NA}}(\mathcal{H}_{|\mathcal{X}\cap D_1 \cap D_2 \cap \cdots \cap D_{n-k}})\nonumber \\
+&\lambda \bigl(I^{\mathrm{NA}}(\mathcal{H}_{| \mathcal{X}\cap D_1 \cap D_2 \cap \cdots \cap D_{n-k}})-(k+1)J^{\mathrm{NA}}(\mathcal{H}_{| \mathcal{X}\cap D_1 \cap D_2 \cap \cdots \cap D_{n-k}})\bigr) \nonumber
\end{align}
if $J^{\mathrm{NA}}(\mathcal{H}_{| \mathcal{X}\cap D_1 \cap D_2 \cap \cdots \cap D_{n-k+1}}),\cdots ,J^{\mathrm{NA}}(\mathcal{H}_b)=0$ where $D_1,\cdots, D_{n-k}\in |L|$ and $b\in B$ are general.
\end{prop}
\begin{proof}
It follows from Lemma \ref{dfnf} that
\[
J^{\mathrm{NA}}(\mathcal{H}_{|\mathcal{X}\cap D_1 \cap D_2 \cap \cdots \cap D_{n-k}})=-E^{\mathrm{NA}}(\mathcal{H}_{|\mathcal{X}\cap D_1 \cap D_2 \cap \cdots \cap D_{n-k}})
\]
and if $J^{\mathrm{NA}}(\mathcal{H}_{| \mathcal{X}\cap D_1 \cap D_2 \cap \cdots \cap D_{n-k+1}}),\cdots ,J^{\mathrm{NA}}(\mathcal{H}_b)=0$ for general $D_j$ and $b$, we have
\begin{align*}
(H^{m+k}\cdot L^{n-k})&R_{\Delta\cap D_1 \cap D_2 \cap \cdots \cap D_{n-k}}^{\mathrm{NA}}(\mathcal{H}_{|\mathcal{X}\cap D_1 \cap D_2 \cap \cdots \cap D_{n-k}})\\
=&p^*K_{(X,\Delta)|_ {\mathcal{X}\cap D_1 \cap D_2 \cap \cdots \cap D_{n-k}}}\cdot (\mathcal{H}_{|\mathcal{X}\cap D_1 \cap D_2 \cap \cdots \cap D_{n-k}})^{m+k}\\
= &p^*(\lambda\, H-f^*L_0)\cdot (\mathcal{H}_{|\mathcal{X}\cap D_1 \cap D_2 \cap \cdots \cap D_{n-k}})^{m+k}\\
= &\lambda\, p^*H\cdot (\mathcal{H}_{|\mathcal{X}\cap D_1 \cap D_2 \cap \cdots \cap D_{n-k}})^{m+k},
\end{align*}
where $p:\mathcal{X}\to X$ is the composition of canonical morphisms $\mathcal{X}\to X_{\mathbb{A}^1}$ and $X_{\mathbb{A}^1}\to X$. The last equality follows from Lemma \ref{melem} below inductively. Therefore, we can see as \cite[Lemma 7.25]{BHJ},
\begin{align*}
p^*H\cdot (\mathcal{H}_{|\mathcal{X}\cap D_1 \cap D_2 \cap \cdots \cap D_{n-k}})^{m+k}=(H^{m+k}\cdot L^{n-k})\Bigl(&I^{\mathrm{NA}}(\mathcal{H}_{|\mathcal{X}\cap D_1 \cap D_2 \cap \cdots \cap D_{n-k}})\\
&-(m+k+1)J^{\mathrm{NA}}(\mathcal{H}_{|\mathcal{X}\cap D_1 \cap D_2 \cap \cdots \cap D_{n-k}})\Bigr)
\end{align*}
and $S(X_b,\Delta_b,H_b)=-m\lambda$. Therefore, the assertion follows from Lemma \ref{induction}. 
\end{proof}
\begin{lem}\label{melem}
Notations as in Proposition \ref{me}. If $J^{\mathrm{NA}}(\mathcal{H}_{|\mathcal{X}\cap  D_1})=0$ for a general ample divisor $D_1$ that is $\mathbb{Q}$-linearly equivalent to $L$, then for any line bundle $M$ on $B$,
\[
\mathcal{H}^{n+m}\cdot p^*f^*M=0.
\]
\end{lem}
\begin{proof}
Let $F=\mathcal{H}-H_{\mathbb{A}^1}$. Since $\mathcal{H}$ is normalized, $F$ is an effective $\mathbb{Q}$-Cartier $\mathbb{Q}$-divisor and note that $ \mathcal{H}|_{F}$ and $H_{\mathbb{P}^1}|_F$ are nef. By assumption, we have
\begin{align*}
0=\mathcal{H}^{n+m}\cdot p^*f^*L=-\sum_{i=0}^{n+m-1} F\cdot \mathcal{H}^{n+m-1-i}\cdot H_{\mathbb{P}^1}^{i}\cdot p^*f^*L
\end{align*}
since $H_{\mathbb{P}^1}^{n+m}\cdot p^*f^*L=0$. On the other hand, $p^*f^*L|_{F}$ is also nef and hence $F\cdot \mathcal{H}^{n+m-1-i}\cdot H_{\mathbb{P}^1}^{i}\cdot p^*f^*L= 0$ for $0\le i\le n+m-1$. Take $k\gg0$ such that $kL\pm M$ are very ample on $B$. Then $\pm F\cdot \mathcal{H}^{n+m-1-i}\cdot H_{\mathbb{P}^1}^{i}\cdot p^*f^*M=F\cdot \mathcal{H}^{n+m-1-i}\cdot H_{\mathbb{P}^1}^{i}\cdot p^*f^*(kL\pm M)\ge 0$. Therefore,
\[
\sum_{i=0}^{n+m-1} F\cdot \mathcal{H}^{n+m-1-i}\cdot H_{\mathbb{P}^1}^{i}\cdot p^*f^*M=-\mathcal{H}^{n+m}\cdot p^*f^*M=0.
\]
Hence, we obtain Lemma \ref{melem} and Proposition \ref{me}.
\end{proof}

Now, we can show that klt Calabi-Yau fibrations are $\mathfrak{f}$-stable as the well-known theorem in K-stability \cite[Theorem 4.1]{OS}.
 
\begin{thm}\label{kds}
Let $f:(X,\Delta,H)\to (B,L)$ be a polarized algebraic fiber space pair with a line bundle $L_0$ on $B$. Suppose that $(X,\Delta,H)$ is a klt polarized pair and $K_X+\Delta\equiv f^*L_0$. Then $f$ is $\mathfrak{f}$-stable.
\end{thm}
\begin{proof}
Let $(\mathcal{X},\mathcal{H})$ be a normal semi fibration degeneration for $f$ (cf., Definition \ref{dedede}), $m=\mathrm{rel.dim}\, f$ and $n=\mathrm{dim}\,B$. By replacing $H$ by $H+jL$ for $j\gg0$, we may assume that $H$ is ample and $(\mathcal{X},\mathcal{H})$ is a normal semiample test configuration normalized with respect to the central fiber. We will prove that if $(\mathcal{X}\cap D_1\cap\cdots\cap D_{n-k},\mathcal{H}_{|\mathcal{X}\cap D_1\cap\cdots\cap D_{n-k}})$ is trivial and $(\mathcal{X}\cap D_1\cap\cdots\cap D_{n-k-1},\mathcal{H}_{|\mathcal{X}\cap D_1\cap\cdots\cap D_{n-k-1}})$ is not trivial for general divisor $D_i\in |ML|$ for $M\gg 0$, then $W_{k+1}^{\Delta}(\mathcal{X},\mathcal{H})>0$ by induction on $k$. Suppose that $(\mathcal{X}\cap D_1\cap\cdots\cap D_{n-k+1},\mathcal{H}_{|\mathcal{X}\cap D_1\cap\cdots\cap D_{n-k+1}})$ is trivial for general divisors $D_i\in |ML|$ for $M\gg 0$. By Proposition \ref{me}
\[
\binom{n+m}{n-k}^{-1}M^{-(n-k)}W^{\Delta}_k(\mathcal{X},\mathcal{H})= (H^{m+k}\cdot L^{n-k})H_{\Delta\cap D_1 \cap D_2 \cap \cdots \cap D_{n-k}}^{\mathrm{NA}}(\mathcal{H}_{|\mathcal{X}\cap D_1 \cap D_2 \cap \cdots \cap D_{n-k}})
\]
for general $D_{n-k}\in |ML|$. It is easy to see by \cite[Proposition 9.16]{BHJ} that if $(X,\Delta)$ is klt, then $(X\cap D_1\cap \cdots\cap D_{n-k},\Delta\cap D_1\cap \cdots\cap D_{n-k})$ is klt for general divisors $D_i\in |ML|$ and hence $H_{\Delta\cap D_1 \cap D_2 \cap \cdots \cap D_{n-k}}^{\mathrm{NA}}(\mathcal{H}_{|\mathcal{X}\cap D_1 \cap D_2 \cap \cdots \cap D_{n-k}})>0$ unless $(\mathcal{X}\cap D_1\cap\cdots\cap D_{n-k},\mathcal{H}_{|\mathcal{X}\cap D_1\cap\cdots\cap D_{n-k}})$ is the trivial test configuration.
\end{proof}

Next, we discuss \cite[Corollary 2.35]{DS4}. Dervan and Sektnan conjectured that fibrations whose all fiber is K-stable are fibration stable in the original sense \cite{DS}. They also conjectured the following,
\begin{conj}[Dervan-Sektnan \cite{DS4}]\label{conjjnoc}
Let $f:(X,H)\to (B,L)$ be a polarized smooth fibration such that any fiber is K-stable. If $\mathrm{Aut}(X_b,H_b)=0$ for any $b\in B$, $f$ is fibration stable in the original sense \cite{DS}.
\end{conj}
 In fact, they proved such fibrations are fibration stable in the sense of \cite{DS4}. However, there exists a following singular example for $\mathfrak{f}$-stability.
\begin{ex}\label{lcbaseex}
Suppose that $B$ is a lc surface that has at least one isolated lc singularity (e.g, a projective cone of an elliptic curve) and $C$ is a curve whose genus $g(C)> 2$ with no automorphism other than the identity (cf., \cite{Ba}). Let $X=C\times B$ be an algebraic fiber space over $B$. Note that $X$ has the relative canonical ample divisor over $B$. Let $L$ be an ample line bundle on $B$ and $H=K_{X/B}+rL$ for sufficiently large $r>0$. Since $B$ has finite singularities, $X$ has lc centers of codimension $2$. Let $b_0$ be one of lc singularities on $B$. Take the deformation to the normal cone $(\mathcal{X},\mathcal{H})$ of $C\times {b_0}$. We can easily see that $W_0(\mathcal{X},\mathcal{H})=W_1(\mathcal{X},\mathcal{H})=0$ by Proposition \ref{cal}. We can compute $W_2(\mathcal{X},\mathcal{H})$ as follows. Note that $H^{\mathrm{NA}}(\mathcal{X},\mathcal{H})=0$ and $R^{\mathrm{NA}}(\mathcal{X},\mathcal{H})+S(X_b,K_{X_b})E^{\mathrm{NA}}(\mathcal{X},\mathcal{H})= I^{\mathrm{NA}}(\mathcal{X},\mathcal{H})-3J^{\mathrm{NA}}(\mathcal{X},\mathcal{H})$ by Proposition \ref{me}. Let $E=H_{\mathbb{P}^1}-\mathcal{H}$ be the exceptional divisor whose center has codimension $3$ in $X\times \mathbb{P}^1$. We can see $H_{\mathbb{P}^1}^4=H_{\mathbb{P}^1}^3\cdot E=H_{\mathbb{P}^1}^2\cdot E^2=0$. Then, we have
\begin{align*}
V(H)I^{\mathrm{NA}}(\mathcal{X},\mathcal{H})&= E\cdot \mathcal{H}^3=3E^3\cdot H_{\mathbb{P}^1}-E^4, \\
-V(H)J^{\mathrm{NA}}(\mathcal{X},\mathcal{H})&=\frac{1}{4}\mathcal{H}^4 =-E^3\cdot H_{\mathbb{P}^1}+\frac{1}{4}E^4.
\end{align*}
Therefore,
\[
V(H)(I^{\mathrm{NA}}(\mathcal{X},\mathcal{H})-3J^{\mathrm{NA}}(\mathcal{X},\mathcal{H}))=-\frac{1}{4}E^4.
\]
Since $E$ is the pullback of the exceptional divisor on the blow-up of $B\times \mathbb{P}^1$, we have $E^4=0$. Hence, $(X,H)\to (B,L)$ is not $\mathfrak{f}$-stable. It also follows that Conjecture \ref{conjjnoc} does not hold in the original sense or for $\mathfrak{f}$-stability.
\end{ex}
 
 \section{Singularities of semistable algebraic fiber spaces}\label{Dsing}
 
 In this section, we see that $\mathfrak{f}$-semistability implies that $X$ has only lc singularities similarly to \cite{GIToda}. Moreover, we obtain Theorem \ref{genFanofib}, which states that K-semistable Fano fibrations have no horizontal non-klt singularities. In the latter part of this section, we consider the case when $X$ is non-normal.
  
 \subsection{Normal case}
We prove the following generalizations of \cite[Theorems 1.2, 1.3]{GIToda} in terms of $\mathfrak{f}$-stability respectively:
 
\begin{thm}\label{singular}
Let $f:(X,\Delta,H)\to (B,L)$ be a polarized algebraic fiber space pair. If $f $ is $\mathfrak{f}$-semistable, $(X,\Delta)$ has at most lc singularities.
\end{thm}

\begin{thm}\label{Fanofib}
Let $f:(X,H)\to (B,L)$ be a flat polarized algebraic fiber space. Suppose that there exist $\lambda\in \mathbb{Q}_{>0}$ and a line bundle $L_0$ on $B$ such that $H+f^*L_0\equiv-\lambda\, K_{X}$ and $B$ has only klt singularities. If $f$ is $\mathfrak{f}$-semistable, then $X$ has only klt singularities. In particular, if $f$ is adiabatically K-semistable (Definition \ref{algfibk}), then $X$ has only klt singularities.
\end{thm}

\begin{proof}[Proof of Theorem \ref{singular}]
Suppose that $m=\mathrm{rel.dim}\, f$ and $n=\mathrm{dim}\,B$ and assume that $(X,\Delta)$ is not lc.
First, assume that $\lceil\Delta\rceil$ is reduced. Then, by Theorem 1.1 of \cite{OX}, there exists the lc modification $\hat{X}$ of $(X,\Delta)$. Thus, we can see that there exists a closed subscheme $Z\subset X$ whose Rees valuations $v$ all satisfy $A_{(X,\Delta)}(v)< 0$ (cf. \cite[Proposition 9.9]{BHJ}) and $\pi:\hat{X}\to X$ is the blow up along $Z$. We can also easily check that the normalization of the deformation to the normal cone $\mathcal{X}$ of $Z$ is a fibration degeneration for $f$. Let $\rho :\mathcal{X}\to X\times \mathbb{A}^1$ be the canonical projection, $E$ be the inverse image of $Z\times\{0\}$ and $\mathcal{H}_\epsilon=\rho^*H_{\mathbb{A}^1}-\epsilon E$ for sufficiently small $\epsilon>0$. We may assume that $H$ is ample and there exists a positive constant $\epsilon_0>0$ such that $(\mathcal{X},\mathcal{H}=\mathcal{H}_{\epsilon_0})$ is a normal ample test configuration by replacing $H$ by $H+jL$ for $j\gg0$ (cf., Lemma \ref{trlem}). Hence, $(\mathcal{X},\mathcal{H}_\epsilon)$ is an ample test configuration for $0<\epsilon<\epsilon_0$. Note that $(\mathcal{X},\mathcal{H}_\epsilon)$ is normalized with respect to the central fiber. Now, we prove that for sufficiently small $0<\epsilon<\epsilon_0$, there exists $k$ such that $W^{\Delta}_k(\mathcal{X},\mathcal{H}_\epsilon)<0$ but $W^{\Delta}_i(\mathcal{X},\mathcal{H}_\epsilon)=0$ for $i<k$. Take $D_1,D_2,\cdots ,D_n\in|ML|$ general for sufficiently divisible integer $M>0$. By replacing $L$ by $ML$, we may assume that $M=1$ as in Lemma \ref{induction}. Then there exists an integer $k$ that is maximal among integers $l$ such that $(\mathcal{X}\cap D_1\cap D_2\cap \cdots\cap D_{n-l+1},\mathcal{H}|_{\mathcal{X}\cap D_1\cap D_2\cap \cdots\cap D_{n-l+1}})$ is trivial. Therefore, $W^{\Delta}_i(\mathcal{X},\mathcal{H}_\epsilon)=0$ for $i<k$ and $J^\mathrm{NA}(\mathcal{H}|_{\mathcal{X}\cap D_1\cap D_2\cap \cdots\cap D_{n-k+1}})=\cdots =J^\mathrm{NA}(\mathcal{H}_b)=0$. By Lemma \ref{induction},
\begin{align*}
W^{\Delta}_k&(\mathcal{X},\mathcal{H}_\epsilon) = \binom{n+m}{n-k} \left(K^{\mathrm{log}}_{(\mathcal{X},\rho^{-1}_* \Delta_{\mathbb{P}^1})/\mathbb{P}^1}\cdot L^{n-k} \cdot \mathcal{H}_\epsilon^{m+k}+\frac{S(X_b,\Delta_b,H_b)}{m+k+1}\mathcal{H}_\epsilon^{m+k+1}\cdot L^{n-k}\right)\\
=&\binom{n+m}{n-k}(L^{n-k}\cdot H^k)\bigl(H_{\Delta\cap D_1\cap D_2\cap \cdots\cap D_{n-k}}^\mathrm{NA}(\mathcal{X}\cap D_1\cap D_2\cap \cdots\cap D_{n-k},\mathcal{H}_\epsilon|_{\mathcal{X}\cap D_1\cap D_2\cap \cdots\cap D_{n-k}})\\
&+R_{\Delta\cap D_1\cap D_2\cap \cdots\cap D_{n-k}}^\mathrm{NA}(\mathcal{X}\cap D_1\cap D_2\cap \cdots\cap D_{n-k},\mathcal{H}_\epsilon|_{\mathcal{X}\cap D_1\cap D_2\cap \cdots\cap D_{n-k}})\\
&+S(X_b,\Delta_b,H_b)E^\mathrm{NA}(\mathcal{X}\cap D_1\cap D_2\cap \cdots\cap D_{n-k},\mathcal{H}_\epsilon|_{\mathcal{X}\cap D_1\cap D_2\cap \cdots\cap D_{n-k}})\bigr).
\end{align*}
If $E\cap D_1\cap D_2\cap \cdots\cap D_{n-k}= \emptyset$, then $(\mathcal{X}\cap D_1\cap D_2\cap \cdots\cap D_{n-k},\mathcal{H}_\epsilon|_{\mathcal{X}\cap D_1\cap D_2\cap \cdots\cap D_{n-k}})$ is almost trivial and this contradicts to the maximality of $k$. Thus we may assume that $E\cap D_1\cap D_2\cap \cdots\cap D_{n-k}\ne \emptyset$.
As in the proof of Proposition 9.12 of \cite{BHJ}, let $\{F_i\}_{i=1}^q$ be the set of prime exceptional divisors of $\hat{X}\cap D_1\cap D_2\cap \cdots\cap D_{n-k}\to X\cap D_1\cap D_2\cap \cdots\cap D_{n-k}$ and suppose that $\min_{1\le i\le q}\mathrm{codim}_{X\cap D_1\cap D_2\cap \cdots\cap D_{n-k}}\rho(F_i)=r$. Then,
\begin{align*}
R_{\Delta\cap D_1\cap D_2\cap \cdots\cap D_{n-k}}^\mathrm{NA}(\mathcal{X}\cap D_1\cap D_2\cap \cdots\cap D_{n-k},\mathcal{H}_\epsilon|_{\mathcal{X}\cap D_1\cap D_2\cap \cdots\cap D_{n-k}})&=O(\epsilon^{r+1}) \\
E^\mathrm{NA}(\mathcal{X}\cap D_1\cap D_2\cap \cdots\cap D_{n-k},\mathcal{H}_\epsilon|_{\mathcal{X}\cap D_1\cap D_2\cap \cdots\cap D_{n-k}})&=O(\epsilon^{r+1})
\end{align*}
since the positive  non-Archimedean metric $(\mathcal{X}\cap D_1\cap D_2\cap \cdots\cap D_{n-k},\mathcal{H}_\epsilon|_{\mathcal{X}\cap D_1\cap D_2\cap \cdots\cap D_{n-k}})$ is also normalized with respect to the central fiber. On the other hand, there exists a positive constant $T>0$ such that
\[
H_{\Delta\cap D_1\cap D_2\cap \cdots\cap D_{n-k}}^\mathrm{NA}(\mathcal{X}\cap D_1\cap D_2\cap \cdots\cap D_{n-k},\mathcal{H}_\epsilon|_{\mathcal{X}\cap D_1\cap D_2\cap \cdots\cap D_{n-k}})=-T\epsilon^r+O(\epsilon^{r+1}) 
\]
if $A_{(X\cap D_1\cap D_2\cap \cdots\cap D_{n-k},\Delta\cap D_1\cap D_2\cap \cdots\cap D_{n-k})}(F_i)<0$ for all $F_i$ such that $\mathrm{codim}_{X\cap D_1\cap D_2\cap \cdots\cap D_{n-k}}\rho(F_i)=r$ by Theorem \ref{bhj48}. Furthermore, we can take so general $D_i$'s that $\hat{X}\cap D_1\cap D_2\cap \cdots\cap D_{n-k}$ are normal. Then, the theorem when $\lceil \Delta\rceil$ is reduced holds by the following claim:
\begin{claim}
If a divisorial valuation $v$ of $X\cap D_1\cap D_2\cap \cdots\cap D_{n-k}$ is one of the Rees valuations of $Z\cap D_1\cap D_2\cap \cdots\cap D_{n-k}$, then $A_{(X\cap D_1\cap D_2\cap \cdots\cap D_{n-k},\Delta\cap D_1\cap D_2\cap \cdots\cap D_{n-k})}(v)<0$.
\end{claim}
Now, we prove the claim. First, note that $\hat{X}\cap D_1\cap D_2\cap \cdots\cap D_{n-k}$ is the (normalized) blow up along $Z\cap D_1\cap D_2\cap \cdots\cap D_{n-k}$. In fact, if $0<k< n$, we can choose $D_1, D_2, \cdots, D_{n-k}$ so general that $\hat{X}\cap D_1\cap D_2\cap \cdots\cap D_{n-k}$ is normal and irreducible. Then we can see that $\hat{X}\cap D_1\cap D_2\cap \cdots\cap D_{n-k}$ is isomorphic to the blow up of $X\cap D_1\cap D_2\cap \cdots\cap D_{n-k}$ along $Z\cap D_1\cap D_2\cap \cdots\cap D_{n-k}$ by \cite[Corollary 7.15]{Ha}. On the other hand, suppose that $k=0$. Then $X\cap D_1\cap D_2\cap \cdots\cap D_{n-1}$ has the normal and connected generic geometric fiber over $D_1\cap D_2\cap \cdots\cap D_{n-1}$. It is easy to see that $g:X\cap D_1\cap D_2\cap \cdots\cap D_{n-1}\to D_1\cap D_2\cap \cdots\cap D_{n-1}$ satisfies that $g_*\mathcal{O}_{X\cap D_1\cap D_2\cap \cdots\cap D_{n-1}}\cong \mathcal{O}_{ D_1\cap D_2\cap \cdots\cap D_{n-1}}$. Since $\hat{X}\cap D_1\cap D_2\cap \cdots\cap D_{n-1}$ is birational to $X\cap D_1\cap D_2\cap \cdots\cap D_{n-1}$, $\hat{X}\cap D_1\cap D_2\cap \cdots\cap D_{n-1}\to D_1\cap D_2\cap \cdots\cap D_{n-1}$ has also connected fibers by the Zariski main theorem \cite[III Corollary 11.4]{Ha}. Therefore, $\hat{X}_b$ is normal and connected for general $b\in B$ and hence $\hat{X}_b\to X_b$ is the blowing up along $X_b\cap Z$. Thus we can conclude that the set of Rees valuations of $Z\cap D_1\cap D_2\cap \cdots\cap D_{n-k}$ coincides with the set of irreducible components of $E_i\cap D_1\cap D_2\cap \cdots\cap D_{n-k}$ for any irreducible component $E_i$ of $\pi$-exceptional divisors.
 
Next, for instance, we show the claim when $k=1$. Since $X\cap D_1$ is a normal Cartier divisor and $$K_{\hat{X}}+\hat{X}\cap D_1=\pi^*(K_X+\Delta+X\cap D_1)+\sum_i (A_{(X,\Delta)}(E_i)-1)E_i$$ where $A_{(X,\Delta)}(E_i)<0$, we have by the adjunction formula (cf., \cite[16.3, 16.4, 17.2]{K+}), $$K_{\hat{X}\cap D_1}=\pi_{|X\cap D_1}^*(K_{X\cap D_1}+\Delta\cap D_1)+\sum_i (A_{(X,\Delta)}(E_i)-1)(E_i\cap D_1).$$ Indeed, note that $E_i\cap D_1$'s are reduced divisors (maybe reducible) and any $\pi_{|D_1}$-exceptional divisor coincides with one of irreducible components of $E_i\cap D_1$ by Bertini's theorem. We also remark that each $E_i$ is not a $\mathbb{Q}$-Cartier but Weil divisor. In general, the above adjunction formula does not hold for any $D_1$. However, we can choose sufficiently general $D_1$ that each $E_i$ is Cartier at codimension 1 points of $\hat{X}\cap D_1$. Then, $E_i\cap D_1=j^wE_i$ on $\hat{X}\cap D_1$ where $j:\hat{X}\cap D_1\to\hat{X}$. Here, $j^wE_i$ is the extension of $(j|_{D_1\cap U})^*(E_i)|_{U}$ to $\hat{X}\cap D_1$ where $U$ is an open subset of $\hat{X}$ such that $\mathrm{codim}_{\hat{X}\cap D_1}(\hat{X}\cap D_1\setminus U)\ge 2$ and $(E_i)|_{U}$ is Cartier (see \cite[16.3]{K+}). Therefore, the above adjunction formula holds and hence we have 
\[
A_{(X\cap D_1, \Delta\cap D_1)}(E_i\cap D_1)=A_{(X,\Delta)}(E_i)<0.
\]
 We also check that $E_i\cap D_1$ is exceptional when we take so general $D_1$. Therefore, the claim holds when $k=1$. The assertion when $k>1$ follows similarly to the case when $k=1$ inductively. Thus, we can conclude that the claim holds.
 
Finally, assume that $\Delta=\sum a_iF_i$ for some $a_i>1$ where $F_i$ are distinct irreducible components. Fix $F_j$ such that $a_j>1$. Let $(\mathcal{X},\mathcal{H}_\epsilon)$ be the deformation to the normal cone of $F_j$ and $\mathcal{H}_\epsilon=H_{\mathbb{P}^1}-\epsilon E$ where $E$ is the exceptional divisor for sufficiently small $\epsilon>0$. Let $k$ be maximal among integers $l$ such that $(\mathcal{X}\cap D_1\cap D_2\cap \cdots\cap D_{n-l+1},\mathcal{H}_{\epsilon|_{\mathcal{X}\cap D_1\cap D_2\cap \cdots\cap D_{n-l+1}}})$ is trivial. Then by the argument in the previous paragraph, there exists a positive constant $T>0$ such that 
\[
H_{\Delta\cap D_1\cap D_2\cap \cdots\cap D_{n-k}}^\mathrm{NA}(\mathcal{X}\cap D_1\cap D_2\cap \cdots\cap D_{n-k},\mathcal{H}_\epsilon|_{\mathcal{X}\cap D_1\cap D_2\cap \cdots\cap D_{n-k}})=-T\epsilon+O(\epsilon^2)
\]
  since the Rees valuations $v\ne\mathrm{ord}_{F_j}$ of $Z$ have the centers $C$ on $X$ such that $\mathrm{codim}_XC\ge 2$ even if $F_j$ is not a $\mathbb{Q}$-Cartier divisor. On the other hand, 
\begin{align*}
R_{\Delta\cap D_1\cap D_2\cap \cdots\cap D_{n-k}}^\mathrm{NA}(\mathcal{X}\cap D_1\cap D_2\cap \cdots\cap D_{n-k},\mathcal{H}_\epsilon|_{\mathcal{X}\cap D_1\cap D_2\cap \cdots\cap D_{n-k}})&=O(\epsilon^{2}) \\
E^\mathrm{NA}(\mathcal{X}\cap D_1\cap D_2\cap \cdots\cap D_{n-k},\mathcal{H}_\epsilon|_{\mathcal{X}\cap D_1\cap D_2\cap \cdots\cap D_{n-k}})&=O(\epsilon^{2}).
\end{align*}
 Thus, $W_k^\Delta(\mathcal{X},\mathcal{H}_\epsilon)<0$ for sufficiently small $\epsilon>0$. We complete the proof.
\end{proof}

To show Theorem \ref{Fanofib}, we need the following,
 
\begin{de}\label{fibertype}
Let $f:X\to B$ be an algebraic fiber space. An irreducible closed subset $Z$ of $X$ is of {\it fiber type} if $\mathrm{codim}_XZ\le\mathrm{codim}_Bf(Z)$.
\end{de}
 
Moreover, we prove the following generalization of Theorem \ref{Fanofib}:
 
\begin{thm}\label{genFanofib}
Let $f:(X,\Delta,H)\to (B,L)$ be a polarized algebraic fiber space pair. Suppose that there exist $\lambda\in \mathbb{Q}_{>0}$ and a line bundle $L_0$ on $B$ such that $H+f^*L_0\equiv-\lambda\, (K_{X}+\Delta)$, and $f$ is $\mathfrak{f}$-semistable. Then $(X,\Delta)$ is lc and any lc-center of $(X,\Delta)$ is of fiber type.
\end{thm}
 
\begin{claim}
Theorem \ref{genFanofib} implies Theorem \ref{Fanofib}.
\end{claim}
 
\begin{proof}[Proof of Claim]
Note that subsets of non-fiber type of $X$ correspond to irreducible closed subsets that do not contain any irreducible component of any fiber of $f$ since $f$ is flat. Due to Theorem \ref{genFanofib}, it suffices to show that any non-klt center of $X$ does not contain any irreducible component of $X_b=f^{-1}(b)$ fiber of $f$ over $b\in B$ if $B$ is klt. This fact immediately follows from Proposition \ref{proplem} and Lemma \ref{lemprop} below. The last assertion of Theorem \ref{Fanofib} follows from the fact that adiabatic K-semistability implies $\mathfrak{f}$-semistability (see Remark \ref{trrem}).
\end{proof}
 
\begin{lem}\label{lemprop}
Let $f:(X,H)\to (B,L)$ be a smooth surjective morphism. If $B$ has only klt singularities, so does $X$.
\end{lem}
\begin{proof}
Let $\pi:B'\to B$ be a resolution of singularities of $B$ and $(X',\Delta')=(X\times _{B}B',X\times _{B}\Delta )$ where $\Delta=\mathrm{Ex}(\pi)$ is snc. Note that $(X',\Delta')$ is log smooth since $f':X'\to B'$ induced by $f$ is also a smooth morphism whose all fiber is connected. Note also that for any two prime divisors $D\ne D'$ on $B'$, $f'^{*}(D)=\lfloor f'^{*}(D) \rfloor$ and $f'^{*}(D)$ and $f'^{*}(D')$ have no common component due to the property of $f'$. Let $\mu :X'\to X$ be the canonical projection respectively. Since $f'$ is the base change of $f$, $\mu^*\Omega _{X/B}=\Omega _{X'/B'}$ and hence $\mu^*K_{X/B}=K_{X'/B'} $. On the other hand, $K_{B'}+F=\pi^*K_B+E$ where $E$ and $F$ are effective divisors that have no common components and $\lfloor F\rfloor =0$. Therefore,
$$
K_{X'}+f'^*(F)=\mu^*K_{X}+f'^*(E)
$$
where $f'^*(F)$ and $f'^*(E)$ are effective divisors that have no common components and $\lfloor f'^*(F)\rfloor =0$. Thus, $X$ has only klt singularities. 
\end{proof}

\begin{prop}\label{proplem}
Let $f:(X,H)\to (B,L)$ be a flat polarized algebraic fiber space whose all fiber is connected and reduced. If $B$ has only klt singularities, $X$ has no lc center whose support contains any irreducible components of $X_b$ for $b\in B$.
\end{prop}
\begin{proof}
Note that a general point of $X_b$ is smooth for $b\in B$. Since $f$ is faithfully flat, we conclude that there exists a closed subset $Z$ such that $f$ is smooth on $X\setminus Z$ and $Z$ contains no component of $X_b$ for any $b\in B$. Hence, the proposition follows from Lemma \ref{lemprop}.
\end{proof}

In the proof of \cite[Theorem 1.3]{GIToda}, Odaka applied \cite[1.4.3]{BCHM}. To prove Theorem \ref{genFanofib}, it is necessary to blow only up lc centers of non-fiber type and we cannot make use of the same argument directly. Hence, we need the slight modification of the technique developed for proving \cite[Theorem 1.3]{GIToda} and \cite[Proposition 9.9]{BHJ} as follows.
 
\begin{prop}\label{parres}
Let $(X,\Delta)$ be an lc pair and $E_j$ be prime divisors over $X$ (maybe exceptional) such that $A_{(X,\Delta)}(E_j)=0$ for $1\le j\le r$. Then there exists a closed subscheme $Z\subset \bigcup_{j=1}^rc_X(E_j)$ such that $$\emptyset\ne\mathrm{Rees}(Z)\subset\left\{ \frac{\mathrm{ord}_{E_j}}{\mathrm{ord}_{E_j}(Z)}|\mathrm{ord}_{E_j}(Z)\ne0\right\}_{j=1}^r.$$ Furthermore, suppose that one of the following holds:
\begin{itemize}
\item[(1)] $r=1$ and $E_1$ is exceptional over $X$,
\item[(2)] Any $E_j$ is a divisor on $X$.
\end{itemize}
Then, we have $$\mathrm{Rees}(Z)=\left\{ \frac{\mathrm{ord}_{E_j}}{\mathrm{ord}_{E_j}(Z)}\right\}_{j=1}^r.$$
\end{prop}
 
\begin{proof}
Let $f:\tilde{X}\to (X,\Delta)$ be a log resolution such that $E_j$'s are smooth divisors on $\tilde{X}$. Suppose that
\[
K_{\tilde{X}}+f^{-1}_*\Delta+F=f^*(K_X+\Delta)+\sum_i A_{(X,\Delta)}(F_i)F_i,
\]
where $F_i$ are irreducible components of $f$-exceptional divisors and $F=\sum_iF_i$. By assumption, $A_{(X,\Delta)}(F_i)\ge 0$. Then by the proof of \cite[4.1]{Fuj}, we can conclude that the $K_{\tilde{X}}+f^{-1}_*\Delta+F$-MMP with scaling over $X$ terminates with a $\mathbb{Q}$-factorial dlt minimal model $(X_1,(f_1^{-1})_*\Delta+F_1)$.
Here, let $f_1:X_1\to X$ be the structure morphism and $D_1$ be the strict transform of $F$. Note that any $F_i$ such that $A_{(X,\Delta)}(F_i)\ne 0$ is contracted but any $E_j$ is not contracted on $X_1$. Then
\[
K_{X_1}+(f_1)_*^{-1}\Delta+D_1=f_1^*(K_X+\Delta)
\]
Let $D_1'$ be the strict transformation of $\sum_{j=1}^r E_j$ and $D_1''=(f_1)_*^{-1}\Delta+D_1-D_1'$. Since $K_{X_1}+(f_1)_*^{-1}\Delta+D_1\sim_{X,\mathbb{Q}}0$, the $K_{X_1}+D_1''+(1-\delta)D_1'$-MMP with scaling over $X$ terminates with a good minimal model for $0<\delta<1$ due to \cite[Theorem 1.1]{B} or \cite[Theorem 1.6]{HX}. Therefore, there exists the log canonical model $X_2$ of $(X_1,D_1''+(1-\delta)D_1')$ over $X$. Let $f_2:X_2\to X$ be the structure morphism and the strict transformations of $D_1'$ and $D_1''$ on $X_2$ be $D_2'$ and $D_2''$ respectively. Then
\[
K_{X_2}+D_2''+(1-\delta)D_2'=f_2^*(K_X+\Delta)-\delta D_2'
\]
and hence $-D_2'$ is $f_2$-ample. Hence, the exceptional set $\mathrm{Ex}(f_2)\subset D_2'$ and any exceptional divisor other than $E_1,\cdots,E_r$ is contracted. If the condition (2) holds, $D_2'=(f_2)^{-1}_*(\sum_{j=1}^rE_j)\ne 0$ and hence we have the second assertion. On the other hand, some $E_j$ is not contracted on $X_2$. Indeed, if any $E_j$ is contracted, we have $D_2'=0$. Then, $(X,\Delta)\cong(X_2,D_2''+(1-\delta)D_2')$. This contradicts to that $(X_2,D_2''+(1-\delta)D_2')$ is the log canonical model of $(X_1,D_1''+(1-\delta)D_1')$. Therefore, the first assertion holds. Furthermore, we have the second assertion by the first one if (1) holds.
\end{proof}
 
\begin{proof}[Proof of Theorem \ref{genFanofib}]
First, we may assume that $\lambda=1$. Let $\mathrm{dim}\,X=N$ and $\mathrm{dim}\,B=n$. Assume that the theorem failed. In other words, we assume that there exists an $\mathfrak{f}$-semistable polarized algebraic fiber space pair $f:(X,\Delta,H)\to (B,L)$ such that there exists at least one lc center of non-fiber type. By Theorem \ref{singular}, it follows that $(X,\Delta)$ is lc. As in the proof of Theorem \ref{singular}, we may also assume that $H$ is ample and $L$ is very ample. Then there exists a closed subscheme $Z$ whose Rees valuations $v$ have log discrepancies $A_{(X,\Delta)}(v)=0$ and are not of fiber type due to Proposition \ref{parres}. Note that $Z$ is of non-fiber type. Let $(\mathcal{X},\mathcal{H})$ be the normalization of the deformation to the normal cone of $Z$, $E=H_{\mathbb{A}^1}-\mathcal{H}$ be the $X\times \mathbb{A}^1$-antiample exceptional divisor and $\mathcal{H}_\epsilon =H_{\mathbb{A}^1}-\epsilon E$ is a positive non-Archimedean metric normalized with respect to the central fiber. We may assume that $\mathcal{H}$ is $\mathbb{A}^1$-ample. Decompose $E=E_1+E_2+\cdots +E_N$ where each irreducible component $E_i^{(s)}$ of $E_i$ has center $Z_i^{(s)}\subset X\times \{ 0\}$ such that $\mathrm{codim}_XZ_i^{(s)}=i$.

We will prove that there exists $k$ such that $W^{\Delta}_k(\mathcal{X},\mathcal{H}_\epsilon)<0$ and $W^{\Delta}_i(\mathcal{X},\mathcal{H}_\epsilon)=0$ for $i<k$ for sufficiently small $\epsilon>0$ as in the proof of Theorem \ref{singular}. First, we take general $D_i$'s. As in the proof, there exists $k$ such that $\mathcal{X}\cap D_1\cap \cdots\cap D_{n-k}$ is not trivial but $\mathcal{X}\cap D_1\cap \cdots\cap D_{n-k+1}$ is trivial. Similarly to the proof of Claim in Theorem \ref{singular}, we can see all the $Z\cap D_1\cap \cdots\cap D_{i}$ also has Rees valuations $v$ that have log discrepancies $A_{(X\cap D_1\cap \cdots\cap D_{i},\Delta\cap D_1\cap \cdots\cap D_{i})}(v)=0$. Therefore, we may assume that $k=n$ by cutting $X$ and $\mathcal{X}$ by general divisors in $|L|$ by Proposition \ref{me}. Furthermore, we may also assume that the cycle $L\cdot E=0$ and hence $E_1=E_2=\cdots =E_{n-1}=E_n=\emptyset$ by the assumption. Let $r=\min \{ i|E_i\ne \emptyset \}>n$. Note that $W^{\Delta}_n(\mathcal{X},\mathcal{H}_\epsilon)=H_{\Delta}^{\mathrm{NA}}(\mathcal{X},\mathcal{H}_\epsilon)-I^{\mathrm{NA}}(\mathcal{X},\mathcal{H}_\epsilon)+(n+1)J^{\mathrm{NA}}(\mathcal{X},\mathcal{H}_\epsilon)$ and $H_{\Delta}^{\mathrm{NA}}(\mathcal{X},\mathcal{H}_\epsilon)=0$ now by Proposition \ref{me}. Let $\{E^{(s)}_j\}_{s=1}^{t_j}$ be the set of irreducible components of $E_j$. Here, we recall the notation in \cite[\S9.3]{BHJ}. Let also $Z_{E^{(s)}_j}$ be the center of $v_{E^{(s)}_j}$ on $X$ and $F_{E_j^{(s)}}$ be the generic fiber of the induced rational map $E_j^{(s)}\dashrightarrow Z_{E_j^{(s)}}\times \{0\}$. If $E_j=\sum_{s=1}^{t_j}m_sE^{(s)}_j$ for $m_s>0$, then for sufficiently small $\epsilon>0$
\[
a_i^{(j)}(\epsilon)\coloneq\sum_{s=1}^{t_j} m_sE^{(s)}_j\cdot \mathcal{H}_\epsilon^i\cdot H_{\mathbb{P}^1}^{N-i} =\left\{
\begin{split}
\epsilon ^j\sum_{s=1}^{t_j} m_s \mathrm{deg}_{E^{(s)}_j}(\mathcal{X},\mathcal{H})\binom{i}{j}(Z_{E^{(s)}_j}\cdot H^{n-j})+O(\epsilon^{j+1}) \quad \mathrm{for} \, i\ge j \\
0\quad \mathrm{for} \, i< j
\end{split}\right.
\]
where $\mathrm{deg}_{E^{(s)}_j}(\mathcal{X},\mathcal{H})=(F_{E^{(s)}_j}\cdot \mathcal{H}^j)>0$. Note that $ \mathrm{deg}_{E^{(s)}_j}(\mathcal{X},\mathcal{H})(Z_{E^{(s)}_j}\cdot H^{n-j})>0$ by ampleness of $\mathcal{H}_\epsilon$. Then, we have by \cite[Lemma 7.4]{BHJ}
\begin{align*}
(H)^N\,I^{\mathrm{NA}}(\mathcal{H}_\epsilon)=& \epsilon \sum_{j=r}^N a_N^{(j)}(\epsilon) =\epsilon a_N^{(r)}(\epsilon)+O(\epsilon^{r+2}), \\
(H)^N\,J^{\mathrm{NA}}(\mathcal{H}_\epsilon)=& \frac{1}{N+1}\epsilon \sum_{j=r}^N \sum_{i=j}^N a_i^{(j)}(\epsilon)= \frac{1}{N+1}\epsilon \sum_{i=r}^N a_i^{(r)}(\epsilon) +O(\epsilon^{r+2}).
\end{align*}
Therefore, we have
\begin{align*}
(H)^N(I^{\mathrm{NA}}(\mathcal{H}_\epsilon)-&(n+1)J^{\mathrm{NA}}(\mathcal{H}_\epsilon))= \epsilon a_N^{(r)}(\epsilon)- \frac{n+1}{N+1}\epsilon \sum_{i=r}^N a_i^{(r)}(\epsilon) +O(\epsilon^{r+2}) \\
=&\epsilon^{r+1}\sum_sm_s\mathrm{deg}_{E^{(s)}_r}(\mathcal{X},\mathcal{H})(Z_{E^{(s)}_r}\cdot H^{n-r})\left(\binom{N}{r}-\frac{n+1}{N+1}\sum_{i=r}^{N}\binom{i}{r}\right) +O(\epsilon^{r+2}) \\
=&\epsilon^{r+1}\sum_s m_s\mathrm{deg}_{E^{(s)}_r}(\mathcal{X},\mathcal{H})(Z_{E^{(s)}_r}\cdot H^{n-r})\binom{N}{r}\left(1-\frac{n+1}{r+1}\right)+O(\epsilon^{r+2}).
\end{align*}
Since $r>n$, $ I^{\mathrm{NA}}(\mathcal{H}_\epsilon)-(n+1)J^{\mathrm{NA}}(\mathcal{H}_\epsilon)>0$ for sufficiently small $\epsilon>0$. Therefore, $W^{\Delta}_n(\mathcal{X},\mathcal{H}_{\epsilon})<0$ for sufficiently small $\epsilon>0$ and this is a contradiction. Thus, we complete the proof.
\end{proof}
 
\subsection{Non-normal case}\label{nnc}
 
Next, we consider the deminormal case. We will define $\mathfrak{f}$-stability of deminormal algebraic fiber spaces and prove generalizations (Theorems \ref{dingular} and \ref{dFanofib}) of Theorem \ref{singular} and Theorem \ref{Fanofib} in this subsection.
\begin{de}\label{dalgfib}
Let $X$ be a deminormal scheme and $\nu_i:X_i\to X$ be the normalization of irreducible components (compare this with Definition \ref{algfib}). A surjective morphism $f:(X,\Delta)\to B$ of equidimensional reduced schemes is a {\it deminormal algebraic fiber space pair} if $f_i:X_i\to B_i$ is an algebraic fiber space for any $X_i$ where $f_i$ is induced by $f\circ\nu_i$ and $B_i$ is the normalization of $f\circ\nu_i(X_i)$, which is an irreducible component of $B$.
\end{de}
 
\begin{de}[cf., Definition \ref{ds1}]\label{ds2}
Suppose that $f:(X,\Delta,H)\to (B,L)$ is a deminormal polarized algebraic fiber space pair with a boundary $\Delta$. Let $\mathrm{rel.dim}\, f=m$ and $\mathrm{dim}\, B=n$. We remark that we can define (semi) fibration degenerations for $f$ similarly to Definition \ref{dedede}. For any fibration degeneration $(\mathcal{X},\mathcal{H})$ for $f$, we define constants $W^{\Delta}_0(\mathcal{X},\mathcal{H}),\, W^{\Delta}_1(\mathcal{X},\mathcal{H}), \cdots ,W^{\Delta}_n(\mathcal{X},\mathcal{H})$ and a rational function $W^{\Delta}_{n+1}(\mathcal{X},\mathcal{H})(j)$ so that the partial fraction decomposition of $\mathrm{DF}_{\Delta}(\mathcal{X},\mathcal{H}+jL)$ in $j$ is as follows:
\[
V(H+jL)\mathrm{DF}_{\Delta}(\mathcal{X},\mathcal{H}+jL)=W^{\Delta}_{n+1}(\mathcal{X},\mathcal{H})(j) + \sum_{i=0}^n j^iW^{\Delta}_{n-i}(\mathcal{X},\mathcal{H}).
\]
 Then $f$ is called
\begin{itemize}
\item {\it $\mathfrak{f}$-semistable} if $W^{\Delta}_0(\mathcal{X},\mathcal{H})\ge 0$ and $W^{\Delta}_0(\mathcal{X},\mathcal{H})=W^{\Delta}_1(\mathcal{X},\mathcal{H})=\cdots =W^{\Delta}_i(\mathcal{X},\mathcal{H})=0\Rightarrow W^{\Delta}_{i+1}(\mathcal{X},\mathcal{H})\ge0$ for $i=0,1,\cdots ,n-1$ for any fibration degeneration.
\item {\it $\mathfrak{f}$-stable} if $f$ is $\mathfrak{f}$-semistable and $W^{\Delta}_i(\mathcal{X},\mathcal{H})= 0,\, i=0,1,\cdots ,n-1\, \Rightarrow W^{\Delta}_n(\mathcal{X},\mathcal{H})>0$ for any fibration degeneration not almost trivial as a test configuration for $(X,H)$.
\end{itemize}
Note that $f$ is
\begin{itemize}
\item $\mathfrak{f}$-semistable if and only if $W^{\Delta}_0(\mathcal{X},\mathcal{H})\ge 0$ and $W^{\Delta}_0(\mathcal{X},\mathcal{H})=W^{\Delta}_1(\mathcal{X},\mathcal{H})=\cdots =W^{\Delta}_i(\mathcal{X},\mathcal{H})=0\Rightarrow W^{\Delta}_{i+1}(\mathcal{X},\mathcal{H})\ge0$ for $i=0,1,\cdots ,n-1$ for any test configuration $(\mathcal{X},\mathcal{H})$ dominating $X_{\mathbb{A}^1}$ such that $(\mathcal{X},\mathcal{H}+jL)$ is a semiample test configuration for sufficiently large $j$.
\item $\mathfrak{f}$-stable if and only if $f$ is $\mathfrak{f}$-semistable and $W^{\Delta}_i(\mathcal{X},\mathcal{H})= 0,\, i=0,1,\cdots ,n-1\, \Rightarrow W^{\Delta}_n(\mathcal{X},\mathcal{H})>0$ for any non-almost-trivial test configuration $(\mathcal{X},\mathcal{H})$ for $(X,H)$ dominating $X_{\mathbb{A}^1}$ such that $\mathcal{H}+jL$ is semiample for sufficiently large $j$.
\end{itemize}
This follows similarly to Definition \ref{ds1}. If $\Delta=0$, we will denote $W^{\Delta}_i=W_i$.
\end{de}
 
If $(\mathcal{X},\mathcal{H})$ as above is partially normal and the normalization $(\widetilde{\mathcal{X}},\widetilde{\mathcal{H}})$ of $(\mathcal{X},\mathcal{H})$ has reduced central fiber, then
\[
\mathrm{DF}_{\Delta}(\mathcal{X},\mathcal{H}+jL)=M^\mathrm{NA}_{\pi_*^{-1}\Delta+\mathfrak{cond}_{\tilde{X}}}(\widetilde{\mathcal{X}},\widetilde{\mathcal{H}}+jL)
\]
where $\pi:\tilde{X}\to X$ is the normalization and hence it also holds that
\[
W_k^{\Delta}(\mathcal{X},\mathcal{H})=W_k^{\pi_*^{-1}\Delta+\mathfrak{cond}_{\tilde{X}}}(\widetilde{\mathcal{X}},\widetilde{\mathcal{H}})
\]
for $k$ (cf., \cite{Od}). 

We prepare the following to calculate $W_k^{\Delta}$ by taking the normalization,
\begin{de}\label{pppipi}
Let $X=\bigcup_{i=1}^r X_i$ be the irreducible decomposition.
For $1\le i\le r$, let $b_i$ be general point of $f(X_i)$. Let also $\widetilde{X_i}$ be the normalization of $X_i$. A deminormal algebraic fiber space pair $f:(X,\Delta,H)\to (B,L)$ has {\it the same scalar curvature with respect to the fiber of} $f$ if
\[
S\left(\widetilde{(X_i)_{b_i}},(\pi_*^{-1}\Delta+\mathfrak{cond}_{\tilde{X}})|_{\widetilde{(X_i)_{b_i}}},H|_{\widetilde{(X_i)_{b_i}}}\right)=S\left(\widetilde{(X_j)_{b_j}},(\pi_*^{-1}\Delta+\mathfrak{cond}_{\tilde{X}})|_{\widetilde{(X_j)_{b_j}}},H|_{\widetilde{(X_j)_{b_j}}}\right)
\]
for $1\le i<j\le r$.
\end{de}
Then we have the following:
 
\begin{lem}[cf. {\cite[Theorem 6.6]{Hat}}]\label{samescalar}
Notations as in \ref{pppipi}. For $1\le i\le r$, let $b_i$ be general point of $f(X_i)$. If $f:(X,\Delta,H)\to (B,L)$ does not have the same scalar curvature with respect to the fiber of $f$, then $f$ is $\mathfrak{f}$-unstable.
\end{lem}
 
\begin{proof}
Suppose that $\mathrm{dim}\, B=n$ and $\mathrm{rel.dim}\,f=m$. Assume that
\[
S\left(\widetilde{(X_1)_{b_1}},(\pi_*^{-1}\Delta+\mathfrak{cond}_{\tilde{X}})|_{\widetilde{(X_1)_{b_1}}},H|_{\widetilde{(X_1)_{b_1}}}\right)>S\left(\widetilde{(X_j)_{b_j}},(\pi_*^{-1}\Delta+\mathfrak{cond}_{\tilde{X}})|_{\widetilde{(X_j)_{b_j}}},H|_{\widetilde{(X_j)_{b_j}}}\right)
\]
for $2\le j\le r$. Let $B=\bigcup B_k$ be the irreducible decomposition and $B_1=f(X_1)$. Then we obtain as \cite[Theorem 6.6]{Hat}
\begin{align*}
S&\left(\widetilde{(X_1)_{b_1}},(\pi_*^{-1}\Delta+\mathfrak{cond}_{\tilde{X}})|_{\widetilde{(X_1)_{b_1}}},H|_{\widetilde{(X_1)_{b_1}}}\right)>\frac{\sum_{k} (L|_{B_k})^nS(X_{b_k},\Delta_{b_k},H_{b_k})}{L^n}.
\end{align*}
 Let $Z=X_1\cap \bigcup_{i\ge 2}X_i$ and $\mathcal{X}$ be the partially normalization of the blow up of $X_{\mathbb{A}^1}$ along $Z\times\{0\}$ with the exceptional divisor $E$. Let $F$ be the strict transformation of $X_1\times\{0\}$. By taking finite base change via the $d$-th power map of $\mathbb{A}^1$, we may assume that the normalization $\widetilde{\mathcal{X}}$ of $\mathcal{X}$ has the reduced central fiber as in the proof of \cite[Proposition 7.16]{BHJ}. Choose $\eta>0$ such that $-E+\eta F$ is $X_{\mathbb{A}^1}$-ample and let
\[
\mathcal{H}=H_{\mathbb{A}^1}-\epsilon (E-\eta F)
\]
be a polarization of $\mathcal{X}$ for sufficiently small $\epsilon>0$.  Then we can prove that
\begin{align*}
\binom{n+m}{n}^{-1}W^\Delta_0(\mathcal{X},\mathcal{H})&=\sum_k(L|_{B_k})^n\left(K_{(\mathcal{X}_{b_k},\Delta_{b_k})/\mathbb{P}^1}\cdot \mathcal{H}_{b_k}^m+S(X_{b_k},\Delta_{b_k},H_{b_k})\frac{\mathcal{H}_{b_k}^{m+1}}{m+1}\right)\\
&=\epsilon\eta (H|_{X_1,b_1})^m \Biggl(\sum_{k} (L|_{B_k})^nS(X_{b_k},\Delta_{b_k},H_{b_k})\\
&-(L)^nS\left(\widetilde{(X_1)_{b_1}},(\pi_*^{-1}\Delta+\mathfrak{cond}_{\tilde{X}})|_{\widetilde{(X_1)_{b_1}}},H|_{\widetilde{(X_1)_{b_1}}}\right)\Biggr)+O(\epsilon^2)
\end{align*}
for general $b_k\in B_k$ similarly to the proof of \cite[Theorem 6.6]{Hat}. 
\end{proof}
 
 Therefore, if a reducible algebraic fiber space $f:(X,\Delta,H)\to (B,L)$ is $\mathfrak{f}$-semistable, then we can decompose $W^\Delta_k$ as follows,
 
\begin{lem}\label{duseful}
Let $f:(X,\Delta,H)\to (B,L)$ be a deminormal polarized algebraic fiber space pair that has the same scalar curvature with respect to the fiber of $f$, $0\le k\le\mathrm{dim}\, B=n$ and $(\mathcal{X},\mathcal{H})$ be a partially normal semiample test configuration for $(X,H)$ dominating $X_{\mathbb{A}^1}$. Suppose that $H$ is ample, $L$ is very ample and $\mathcal{X}$ has the reduced central fiber. Let $\nu:\tilde{X}\to X$ be the normalization and $\tilde{X}=\bigcup_{i=1}^r \widetilde{X_i}$ be the irreducible decomposition. Let also $\widetilde{\mathcal{X}}$ be the normalization of $\mathcal{X}$ and $\widetilde{\mathcal{X}}=\bigcup_{i=1}^r \widetilde{\mathcal{X}_i}$ be the irreducible decomposition where the indices corresponding to those of $\tilde{X}=\bigcup_{i=1}^r \widetilde{X_i}$. 

If $(\widetilde{\mathcal{X}_i},\mathcal{H}|_{\widetilde{\mathcal{X}_i}})$ is normalized with respect to the central fiber and $(\widetilde{\mathcal{X}_i}\cap D_1\cap\cdots\cap D_{n-j},\mathcal{H}|_{\widetilde{\mathcal{X}_i}\cap D_1\cap\cdots\cap D_{n-j}} )$ is trivial for $0\le j<k$ and for general ample divisors $D_1,\cdots,D_{n-k+1}\in|L|$, then
\[
W^\Delta_k(\mathcal{X},\mathcal{H})=\sum_{i=1}^rW^{(\pi_*^{-1}\Delta+\mathfrak{cond}_{\tilde{X}})|_{\widetilde{X}_i}}_k(\widetilde{\mathcal{X}_i},\mathcal{H}|_{\widetilde{\mathcal{X}_i}}).
\]
\end{lem}
 
\begin{proof}
By the assumption, we have 
\[
\mathrm{DF}_\Delta(\mathcal{X},\mathcal{H}+jL)=M_{(\pi_*^{-1}\Delta+\mathfrak{cond}_{\tilde{X}})}^\mathrm{NA}(\widetilde{\mathcal{X}},\widetilde{\mathcal{H}}+jL)
\]
for any $j$. Moreover, since each $(\widetilde{\mathcal{X}_i},\mathcal{H}|_{\widetilde{\mathcal{X}_i}})$ is normalized, we have
\[
W_k^{(\pi_*^{-1}\Delta+\mathfrak{cond}_{\tilde{X}})}(\widetilde{\mathcal{X}},\widetilde{\mathcal{H}})=\sum_{i=1}^rW^{(\pi_*^{-1}\Delta+\mathfrak{cond}_{\tilde{X}})|_{\widetilde{X}_i}}_k(\widetilde{\mathcal{X}_i},\mathcal{H}|_{\widetilde{\mathcal{X}_i}})
\]
by Lemmas  \ref{induction} and \ref{samescalar}. 
\end{proof}

Note that if $(\mathcal{X},\mathcal{H})$ is a deformation to the normal cone of a closed subscheme $Z$ of $X$ with the exceptional divisor $E$ such that $\mathrm{dim}\,Z<\mathrm{dim}\,X$ and $\mathcal{H}=H_{\mathbb{P}^1}-\epsilon E$ for sufficiently small $\epsilon>0$, then $(\widetilde{\mathcal{X}_i},\mathcal{H}|_{\widetilde{\mathcal{X}_i}})$ is normalized with respect to the central fiber in Lemma \ref{duseful}. Then, we can prove the following,

\begin{thm}\label{dingular}
Let $f:(X,\Delta,H)\to (B,L)$ be a polarized deminormal algebraic fiber space. If $f $ is $\mathfrak{f}$-semistable, $(X,\Delta)$ has at most slc singularities.
\end{thm}
 
\begin{proof}
Assume that $(X,\Delta)$ is not slc but $f $ is $\mathfrak{f}$-semistable. If $\lceil\Delta\rceil$ is not reduced, let $F$ be an irreducible component of $\Delta$ whose coefficient is larger than 1. Then, as in the proof of Theorem \ref{singular}, we can prove that the partially normalization of the deformation to the normal cone $(\mathcal{X},\mathcal{H})$ of $F$ with some polarization $\mathcal{H}$ satisfies that there exists $0\le k\le \mathrm{dim}\, B$ such that
\[
W_k^\Delta (\mathcal{X},\mathcal{H})<0
\]
and
\[
W_i^\Delta (\mathcal{X},\mathcal{H})=0
\]
for $i<k$ by Lemma \ref{samescalar} and Lemma \ref{duseful}. Thus, we may assume that $\lceil\Delta\rceil$ is reduced. By \cite[Corollary 1.2]{OX}, there exists the slc modification $\pi:Y\to X$.
It is easy to see that there exists a closed subscheme $Z$ such that $\pi$ is the blow up along $Z$. Let $\nu:\tilde{X}\to X$ be the normalization and $D=\mathfrak{cond}_{\tilde{X}}$ be the conductor. If $\nu_Y:\tilde{Y}\to Y$ is the normalization of $Y$ and $\tilde{\pi}:\tilde{Y}\to\tilde{X}$ is the induced morphism, then $(\tilde{Y},\Delta_{\tilde{Y}}+\tilde{\pi}_*^{-1}D)$ is the lc modification of $(\tilde{X},\Delta_{\tilde{X}}+D)$ where $\Delta_{\tilde{Y}}=(\nu_Y)^{-1}_*\Delta_Y$ and $\Delta_{\tilde{X}}=\nu^{-1}_*\Delta_X$ (cf., \cite[Lemma 3.1]{OX}). It is easy to see that $\tilde{\pi}$ is the normalized blow up along $\nu^{-1}Z$ and $A_{(\tilde{X}, \Delta_{\tilde{X}}+D)}(v)<0$ for $v\in \mathrm{Rees}(\nu^{-1}Z)$. Let $(\mathcal{X},\mathcal{H}_\epsilon=H_{\mathbb{P}^1}-\epsilon E)$ be the partially normalization of the deformation to the normal cone of $Z$ where $E$ is the exceptional divisor and $\epsilon>0$. Let also $(\widetilde{\mathcal{X}},\widetilde{\mathcal{H}_\epsilon})$ be the normalization of $(\mathcal{X},\mathcal{H}_\epsilon)$. We may assume that the central fiber $\widetilde{\mathcal{X}}_0$ of $\widetilde{\mathcal{X}}$ is reduced by replacing $\widetilde{\mathcal{X}}$ by the partially normalized base change via $\mathbb{A}^1\ni t\mapsto t^d\in \mathbb{A}^1$. Then
\[
M_{(\Delta_{\tilde{X}}+D)}^{\mathrm{NA}}(\widetilde{\mathcal{X}},\widetilde{\mathcal{H}_\epsilon}+jL)=\mathrm{DF}_{(\Delta_{\tilde{X}}+D)}(\widetilde{\mathcal{X}},\widetilde{\mathcal{H}_\epsilon}+jL)=\mathrm{DF}_\Delta(\mathcal{X},\mathcal{H}_\epsilon+jL),
\]
for $j\in\mathbb{Q}$. Therefore,
\[
W_k^{(\Delta_{\tilde{X}}+D)}(\widetilde{\mathcal{X}},\widetilde{\mathcal{H}_\epsilon})=W_k^\Delta(\mathcal{X},\mathcal{H}_\epsilon)
\]
for $0\le k \le \mathrm{dim}\, B$. Hence, by the proof of Theorem \ref{singular} and Lemmas \ref{samescalar} and \ref{duseful}, we can prove that there exists $0\le k\le \mathrm{dim}\, B$ such that
\[
W_k^\Delta (\mathcal{X},\mathcal{H}_\epsilon)<0
\]
and
\[
W_i^\Delta (\mathcal{X},\mathcal{H}_\epsilon)=0
\]
for $i<k$ and for sufficiently small $\epsilon>0$.
\end{proof}
 
 As in Proposition \ref{parres}, we need the following partial resolution.
 
\begin{lem}\label{darres}
Let $(X,\Delta)$ be a projective slc pair and fix an slc center $C$ such that $\mathrm{codim}\,C\ge 2$ (i.e., if $\nu:\tilde{X}\to X$ is the normalization, there exists an lc center $C'$ such that $\nu(C')=C$). Then there exists a closed subscheme $Z$ that satisfies the following conditions.
\begin{enumerate}
\item The reduced structure $\mathrm{red}\,(Z)$ is contained in $C$,
\item $A_{(X,\Delta)}(v)=0$ for any $v\in \mathrm{Rees}(\nu^{-1}Z)$, 
\item There exists at least one valuation $v\in \mathrm{Rees}(\nu^{-1}Z)$ such that the center of $v$ dominates $C$.
\end{enumerate}
\end{lem}
 
\begin{proof}
Fix an ample line bundle $H$ and let $\mathscr{I}$ be the ideal sheaf corresponding to the reduced structure of $C$. Then the linear system $\mathfrak{d}=H^0(X,\mathscr{I}\otimes\mathcal{O}(mH))$ is base point free outside from $C$ for sufficiently large $m>0$. We can choose $D\in \mathfrak{d}$ such that $D$ contains no slc centers other than those contained in $C$ and $ X_i\not\subset \mathrm{supp}(D)$ for any irreducible component $X_i$ of $X$. Let $f:Y\to\tilde{X}$ be a log resolution of $(\tilde{X},\nu^{-1}_*\Delta+\nu^*D+\mathfrak{cond}_{\tilde{X}})$ and a resolution of the base locus of $ \mathfrak{d}$. By replacing $D$ by general one, we may assume that $f^*\nu^*D$ does not contain any prime divisor $E$ on $Y$ such that $A_{(X,\Delta)}(E) (=A_{(\tilde{X},\nu^{-1}_*\Delta+\mathfrak{cond}_{\tilde{X}})}(E))=0$ and $\nu\circ f(E)\not\subset C$ due to the theorem of Bertini. Let $\Delta_Y$ be the $\mathbb{Q}$-divisor satisfies that
\[
K_Y+\Delta_Y=f^*(K_{\tilde{X}}+\nu^{-1}_*\Delta+\mathfrak{cond}_{\tilde{X}}).
\]
Then, since $(Y,\Delta_Y)$ is log smooth and sublc, for non-lc centers $C'$ on $(\tilde{X},\nu^{-1}_*\Delta+\epsilon\nu^*D+\mathfrak{cond}_{\tilde{X}})$, $\nu(C')\subset C$ for sufficiently small rational $\epsilon>0$. Note that $(X,\Delta+\epsilon\nu^*D)$ is not slc along $C$ but $\lceil \epsilon\nu^*D+\Delta\rceil$ is reduced. Thanks to \cite[Corollary 1.2]{OX}, we take the slc modification $g:W\to X$ of $(X,\Delta+\epsilon\nu^*D)$ and there exists a closed subscheme $Z$ such that $g$ is the blow up along $Z$. It is easy to see that $\mathrm{red}\,(\nu^{-1}Z)$ is contained in $C$, $A_{(X,\Delta)}(v)=0$ for any $v\in \mathrm{Rees}(\nu^{-1}Z)$ and there exists at least one valuation $v\in \mathrm{Rees}(Z)$ such that the center of $v$ dominates $C$.
\end{proof}
 
\begin{thm}\label{dFanofib}
Let $f:(X,\Delta,H)\to (B,L)$ be a polarized deminormal algebraic fiber space. Suppose that there exist $\lambda\in \mathbb{Q}_{>0}$ and a line bundle $L_0$ on $B$ such that $H+f^*L_0\equiv-\lambda\, (K_{X}+\Delta)$, and $f$ is $\mathfrak{f}$-semistable. Then, $(X,\Delta)$ is slc and any slc-center of $(X,\Delta)$ is of fiber type.
\end{thm}
 
\begin{proof}
We may assume that $H$ is ample and $L$ is very ample. It follows from Theorem \ref{dingular} that $(X,\Delta)$ is slc. Assume that there exists at least one slc-center or singular locus of codimension $1$ of $(X,\Delta)$ that is of non-fiber type.
 
First, assume that there exists at least one irreducible component of non-fiber type of the conductor subscheme $D=\mathfrak{cond}_X$ of $X$. Let $\nu:\tilde{X}\to X$ be the normalization, $\tilde{D}=\mathfrak{cond}_{\tilde{X}}$ be the conductor divisor. Note that $\tilde{D}$ need not to be $\mathbb{Q}$-Cartier in general. Due to Proposition \ref{parres}, there exists a coherent ideal sheaf $\mathfrak{a}\subset\mathcal{O}_{\tilde{X}}$ such that $\mathrm{Rees}(\mathfrak{a})$ is the set $\{ \frac{\mathrm{ord}_{D_i}}{\mathrm{ord}_{D_i}(\mathfrak{a})}\}$ where $D_i$ are all irreducible components of $\tilde{D}$. Since $\mathfrak{a}\subset\mathcal{O}_{\tilde{X}}(-\tilde{D})$, we can consider $\mathfrak{a}$ to be an ideal sheaf of $X$ and let $Z$ be the closed subscheme of $X$ corresponding to $\mathfrak{a}$. Then consider the partially normalization of the deformation to the normal cone $(\mathcal{X},\mathcal{H}_\epsilon=H_{\mathbb{A}^1}-\epsilon E)$ of $Z$ where $E$ is the exceptional divisor for $\epsilon>0$. We may assume that the central fiber of $\mathcal{X}$ is reduced by the same argument of the proof of Theorem \ref{dingular}. If $\tilde{\mathcal{X}}$ is the normalization of $\mathcal{X}$, then $H_\Delta^{\mathrm{NA}}(\tilde{\mathcal{X}}_b,\widetilde{\mathcal{H}_{\epsilon}|_{\tilde{\mathcal{X}}_b}})=0$ for general $b\in B$ and hence $W_0^{\Delta}(\mathcal{X},\mathcal{H}_\epsilon)<0$. Therefore, we may assume that any singular locus of codimension $1$ of $(X,\Delta)$ is of fiber type. In other words, we may assume that the general fiber is normal. Moreover, the general fiber is also lc since $(X,\Delta)$ is slc.
 
Next, assume that there exists at least one irreducible component of non-fiber type $F$ of $\lfloor\Delta\rfloor$. Let $X_1$ be the irreducible component of $X$ containing $F$. There exists a closed subscheme $Z$ of $X_1$ such that $\mathrm{Rees}(Z)=\{ F\}$ due to Proposition \ref{parres}. For general point $b\in f(X_1)\subset B$, $((X_1)_b,\Delta_{|(X_1)_b})$ is an lc pair and $\mathrm{Rees}(Z_b)=\{ F_b\}$. Therefore, it is easy to see that if $(\mathcal{X},\mathcal{H}_\epsilon)$ is the partially normalization of the deformation of $X$ to the normal cone of $Z$ where $E$ is the inverse image of $Z\times\{0\}$ and $\mathcal{H}_\epsilon=H_{\mathbb{A}^1}-\epsilon E$, then we have $W_0^{\Delta}(\mathcal{X},\mathcal{H}_\epsilon)<0$ for sufficiently small $\epsilon>0$. Thus, we conclude that $\lfloor\Delta\rfloor$ is of fiber type.
 
Finally, let $k=\min\{\mathrm{codim}_Bf(C):$ where $C$ is a slc center of non-fiber type$\}$ and fix a slc center $C$ of non-fiber type such that $\mathrm{codim}_Bf(C)=k$. We may assume that $\mathrm{codim}\,C\ge 2$. By Lemma \ref{darres}, there exists a closed subscheme $Z$ contained in $C$ such that $A_{(X,\Delta)}(v)=0$ for any $v\in \mathrm{Rees}(Z)$ and there exists at least one valuation $v\in \mathrm{Rees}(Z)$ such that the center of $v$ dominates $C$. Take general divisors $D_1,\cdots,D_{n-k}\in |L|$ and cut $X$ by $D_1,\cdots,D_{n-k}$. Now, let $(\mathcal{X},\mathcal{H}_\epsilon)$ be the partially normalization of the deformation of $X$ to the normal cone of $Z$ where $E$ is the exceptional divisor and $\mathcal{H}_\epsilon=H_{\mathbb{A}^1}-\epsilon E$. Then, we will prove that $W_k^{\Delta}(\mathcal{X},\mathcal{H}_\epsilon)<0$ for sufficiently small $\epsilon>0$. Here, we may assume that $\mathrm{dim}\, B=k$ by replacing $X$ by $X\cap D_1\cap\cdots\cap D_{n-k}$. Since $\mathrm{dim}\,f(C)=0$, any center of $v'\in \mathrm{Rees}(Z)$ is of non-fiber type. Hence, we can prove that \[
W_k^\Delta (\mathcal{X},\mathcal{H}_\epsilon)<0
\]
and
\[
W_i^\Delta (\mathcal{X},\mathcal{H}_\epsilon)=0
\]
for $i<k$ and sufficiently small $\epsilon>0$ similarly as Theorem \ref{dingular} and Theorem \ref{genFanofib}.
\end{proof}
 
Adiabatic K-semistability implies $\mathfrak{f}$-semistability. Thus we conclude that the following holds by the previous theorem.
\begin{cor}\label{kei56}
 Let $f:(X,\Delta,H)\to (B,L)$ be an adiabatically K-semistable polarized deminormal algebraic fiber space such that $H+f^*L_0\equiv-\lambda\, (K_{X}+\Delta)$ where $L_0$ is a line bundle on $B$ and $\lambda\in\mathbb{Q}_{>0}$. Then $(X,\Delta)$ is slc and any slc-center of $(X,\Delta)$ is of fiber type.
 \end{cor}

\end{document}